\documentclass[11pt]{amsart}

\RequirePackage[l2tabu,orthodox]{nag}
\usepackage[T1]{fontenc}
\usepackage[utf8]{inputenc}
\usepackage{lmodern}
\usepackage[english]{babel}
\usepackage{color} 
\usepackage[usenames,dvipsnames]{xcolor}
\usepackage[alphabetic]{amsrefs}
\usepackage{amsthm,amssymb,amsmath}
\usepackage[pdftex]{graphicx}
\usepackage{enumerate}
\usepackage[inline]{enumitem}
\usepackage{url}
\usepackage{stmaryrd}

\usepackage{multicol}
\usepackage{mathrsfs}

\usepackage[foot]{amsaddr}

\numberwithin{equation}{section}

\let\oldtocsection=\tocsection

\let\oldtocsubsection=\tocsubsection 

\let\oldtocsubsubsection=\tocsubsubsection

\renewcommand{\tocsection}[2]{\hspace{0em}\oldtocsection{#1}{#2}}
\renewcommand{\tocsubsection}[2]{\hspace{2em}\oldtocsubsection{#1}{#2}}
\renewcommand{\tocsubsubsection}[2]{\hspace{4em}\oldtocsubsubsection{#1}{#2}}

\newtheorem{thm}{Theorem}[section]
\newtheorem{introthm}{Theorem}

\newtheorem*{claim}{Claim}

\newtheorem{claimnum}{Claim}

\newtheorem{lem}[thm]{Lemma}
\newtheorem{prop}[thm]{Proposition}
\newtheorem{cor}[thm]{Corollary}
\newtheorem{intro-cor}[introthm]{Corollary}

\theoremstyle{definition}
\newtheorem{dfn}[thm]{Definition}

\newtheorem*{main question}{Main Question}

\theoremstyle{remark}
\newtheorem{rem}[thm]{Remark}
\newtheorem{ex}[thm]{Example}

\usepackage[margin=3cm]{geometry}

\usepackage{indentfirst}
\usepackage{microtype}

\usepackage[colorlinks=true,linkcolor=blue,citecolor=magenta]{hyperref}

\usepackage{blindtext}

\usepackage{nameref}

\DeclareMathOperator{\Aut}{\mathsf{Aut}}

\DeclareMathOperator{\Fix}{\mathsf{Fix}}
\DeclareMathOperator{\homeo}{\mathsf{Homeo}}
\DeclareMathOperator{\Diff}{\mathsf{Diff}}
\DeclareMathOperator{\Bij}{\mathsf{Sym}}
\DeclareMathOperator{\Fit}{\mathsf{Fit}}
\DeclareMathOperator{\vf}{\mathsf{vf}}

\DeclareMathOperator{\stab}{\mathsf{Stab}}

\DeclareMathOperator{\GL}{\mathsf{GL}}
\DeclareMathOperator{\hor}{\mathfrak{h}}
\newcommand{\wrwr}{\operatorname{\wr\wr}}

\newcommand{\fix}{\operatorname{\mathsf{Fix}}}
\newcommand{\fixphi}{\operatorname{\mathsf{Fix}}^\varphi}
\DeclareMathOperator{\Aff}{\mathsf{Aff}}

\newcommand{\id}{\mathsf{id}}
\newcommand{\R}{\mathbb{R}}
\newcommand{\Q}{\mathbb Q}
\newcommand{\N}{\mathbb N}
\newcommand{\Z}{\mathbb Z}

\newcommand{\Tbb}{\mathbb T}

\newcommand{\Lcal}{\mathcal{L}}
\renewcommand{\setminus}{\smallsetminus}

\DeclareMathOperator{\BS}{\mathsf{BS}}
\DeclareMathOperator{\Br}{\mathsf{Br}}
\DeclareMathOperator{\Homirr}{\mathsf{Hom}_{\mathrm{irr}}}

\DeclareMathOperator{\Int}{\mathsf{Int}}
\renewcommand{\ker}{\operatorname{\mathsf{ker}}}

\renewcommand{\emptyset}{\varnothing}
\newcommand{\treeorder}{\triangleleft}

\newcommand{\treeorderop}{\triangleright}
\newcommand{\treeordereq}{\trianglelefteq}
\newcommand{\treeup}{\curlywedgeuparrow}

\DeclareMathOperator{\supp}{\mathsf{Supp}}

\DeclareMathOperator{\suppphi}{\mathsf{Supp}^\varphi}

\title{Solvable groups and affine actions on the line}
\author[Brum \and Matte Bon \and Rivas \and Triestino]{Joaqu\'in Brum \and Nicol\'as Matte Bon \and Crist\'obal Rivas \and Michele Triestino}
\date{\today}
\keywords{Group actions on the real line, solvable groups, actions on real trees, $C^1$ actions}

\begin{document}
\begin{abstract}
	We prove a structural result for orientation-preserving actions  of finitely generated solvable groups on real intervals, considered up to semi-conjugacy. As applications we obtain new answers to a problem first considered by J. F. Plante, which asks  under which conditions  an action of a solvable group on a real interval is semi-conjugate to an action on the line by affine transformations. We show that this is always the case for actions by $C^1$ diffeomorphisms on closed intervals. For arbitrary actions by homeomorphisms, for which this result is no longer true (as shown by Plante),  we show that a semi-conjugacy to an affine action still exists in a local sense, at the level of germs near the endpoints. Finally for a vast class of solvable groups, including all solvable linear groups, we show that the family of affine actions on the line is robust, in the sense that any  action by homeomorphisms on the line which is sufficiently close to an affine action must be semi-conjugate to an affine action. This robustness fails for general solvable groups, as illustrated by  a counterexample. 
	 
	\smallskip
	
	{\noindent\footnotesize \textbf{MSC\textup{2020}:} Primary 37C85, 20F16, 20E08, 20F60, 57M60. Secondary 37E05, 37B05.}

\end{abstract}

\maketitle

\section{Introduction}
\subsection{Background}
This work is about the study of group actions on one-manifolds. This means, given a group $G$, to try to understand and classify the possible  representations $\varphi\colon G \to \homeo_0(X)$ into the group of orientation-preserving homeomorphisms, or into the group $\Diff^r_0(X)$ of diffeomorphisms of regularity $C^r$ where $X$ is a connected one-manifold (that is, the circle or a real interval).  With no real loss of generality we restrict to actions which have no global fixed points in the interior of $X$, which will be termed \emph{irreducible} (indeed, an action  with global fixed points can be decomposed as a union of irreducible actions on a countable family of invariant intervals).  As usual in dynamics,  actions can be considered up to (topological) conjugacy; however, in the one-dimensional setting it is always possible to obtain many non-conjugate actions from a given one by blowing up orbits as in the classical Denjoy's example. This redundancy can be avoided by considering a weaker equivalence relation, called \emph{semi-conjugacy}, whose definition can be traced back to Ghys \cite{ghysbounded}, although it has been unanimously formalized only recently (see for instance the monograph by Kim, Koberda, and Mj \cite{KKMj}).  When $X, Y$ are non-empty open intervals,  two irreducible actions $\varphi\colon G\to \homeo_0(X)$ and $\psi \colon G \to \homeo_0(Y)$ are semi-conjugate if there is a monotone $G$-equivariant map $h\colon X\to Y$.  We warn the reader that this semi-conjugacy notion is slightly different from the homonymous notion appearing in classical dynamical systems: notably, this is an equivalence relation, which roughly speaking identifies two actions if they have the same behavior on their  minimal closed invariant subsets. (We refer to  \S \ref{sc.preliminary-actions} for more details.) 

The case of abelian groups is closely related to the classical theory of circle homeomorphisms initiated by Poincaré: every irreducible action of a finitely generated abelian group on the line is semi-conjugate to an action by translations. Stronger rigidity holds for actions of higher regularity: for $C^2$ actions, the results of  Denjoy \cite{Denjoy} and Kopell \cite{Kopell}, combined with a theorem of H\"older \cite{Holder}, together imply that an abelian subgroup of $\Diff^2_0([0, 1])$ whose action is irreducible, is topologically conjugate (in restriction to $(0, 1)$) to a group of translations of $\R$. These results are a starting point for many other rigidity results for group actions in $C^r$ regularity, with $r>1$. However, as these are typically based on distortion estimates, they usually fail in $C^1$ regularity, as one can see from the examples of Tsuboi \cite{Tsuboi-Kopell}, generalizing the classical constructions of Denjoy \cite{Denjoy} and Pixton \cite{Pixton}. 

The situation for nilpotent groups is completely analogous to the abelian case.  Plante showed that every irreducible action $\varphi\colon G\to \homeo_0(\R)$ of a finitely generated nilpotent group is semi-conjugate to an action by translations \cite{PlanteMeasure} (this holds more generally  for groups $G$ of subexponential growth, although the existence of non-virtually nilpotent such groups was unknown at the time). Plante and Thurston showed (building on the case of abelian groups) that every nilpotent subgroup of $\Diff_0^2([0, 1])$ is abelian  \cite{Plante-Thurston}; in particular, it is topologically conjugate to a group of translations, as soon as its action on the interval is irreducible.  This is far from being true for $C^1$ actions. Indeed, Farb and Franks \cite{Farb-Franks} showed that every torsion-free finitely generated nilpotent group embeds in  $\Diff^1_0([0, 1])$ (see also the results of Jorquera \cite{Jorquera} and Castro, Jorquera, and Navas \cite{CJN}). Furthermore, Parkhe's result from \cite{Parkhe} implies that every $C^0$ action of a finitely generated nilpotent group on $[0,1]$ can be conjugated into a $C^1$ action.

The next natural class  to consider is that of finitely generated solvable groups, and this is the main focus of this paper. The study of this case  goes back to a seminal paper of Plante \cite{Plante}. A basic source of examples of solvable groups acting on one-manifolds are subgroups of the affine group
\[\Aff(\R)=\{x\mapsto ax+b:  a>0, b\in \R\}.\]
The main problem considered by Plante in \cite{Plante} is the following: under which conditions must an action of a solvable group $\varphi\colon G\to \homeo_0(\R)$ be semi-conjugate to an action  by affine transformations $\psi\colon G\to\Aff(\R)$?

Plante provided an array of sufficient  conditions. He deduced in particular that when $G$ is a finitely generated polycyclic group, every irreducible action $\varphi\colon G \to \homeo_0(\R)$ is semi-conjugate to an affine action.  In contrast, he showed that this fails for general solvable groups, as he constructed a faithful minimal action of the group $\Z \wr \Z$ on the line which is not semi-conjugate to any affine action, see \cite[\S 5]{Plante}.  We recover Plante's construction from a different method that we explain in \S \ref{sec ejemplo lamp lighter}, and call the resulting actions \emph{Plante-like actions} of $\Z\wr\Z$. The idea behind our approach is quite robust and can be easily generalized in multiple ways (see for instance our previous work \cite[Example  8.1.8]{BMRT}). As a consequence, it turns out that the possible actions of general solvable groups on the real line can be wilder than expected. Since Plante's work, a number of papers have addressed the study of solvable group actions on one-manifolds, and the special  role played by affine actions  can be recognized as the main theme, see for instance works of Plante \cite{PlanteContinuous}, Cantwell and Conlon \cite{CC}, Navas \cite{NavasResolubles,NavasAmenable},   Akhmedov \cite{Akhmedov_Z,Akhmedov}, Bonatti, Monteverde, Navas and the third author \cite{BMNR}, Guelman and the third author \cite{GuelmanRivas} (see also work by Chiswell and Kropholler \cite{ChiswellKropholler} and by the third author and Tessera \cite{RivasTessera} for tightly related results on left orders on solvable groups).  
 
\subsection{Results}
In this paper we develop a new approach to study general actions of solvable groups by homeomorphisms of the real line; we actually work under the slightly more general assumption that $G$ is virtually solvable, i.e.\ contains a solvable subgroup of finite index. As an application, we obtain three  results which provide new answers to the problem originally considered by Plante, by relating arbitrary actions to affine actions.

Our first result deals with actions on intervals by diffeomorphisms of class $C^1$. Here and elsewhere, we say for short that an affine action $\psi \colon G\to \Aff(\R)$ is \emph{abelian} if its image is abelian.

\begin{introthm}[$C^1$ actions of solvable groups]\label{mthm.C1}
	Let $\varphi\colon G\to \Diff^1_0([0, 1])$ be an irreducible action of a finitely generated virtually solvable group. Then the restriction of $\varphi$ to $(0, 1)$ is semi-conjugate to an affine action $\psi\colon G\to \Aff(\R)$. Moreover,  the semi-conjugacy is a topological conjugacy, provided $\psi$ is non-abelian. 
	\end{introthm}

\begin{rem}
	The main contribution of the theorem is the existence of the semi-conjugacy. When  $\psi$ is a non-abelian affine action, any $C^1$ action semi-conjugate to $\psi$ is automatically conjugate  to $\psi$ by a result of  Bonatti, Monteverde, Navas, and the third author \cite[\S 4.2]{BMNR}.
\end{rem} 

\begin{rem}
It should be stressed that Theorem \ref{mthm.C1} does not imply that every solvable subgroup of $\Diff^1_0([0, 1])$ is metabelian,  since a semi-conjugacy might not preserve faithfulness of an action; in fact, it is easy to build solvable subgroups of $\Diff^\infty_0([0, 1])$ of arbitrarily large solvability degree by considering iterated wreath products of $\Z$, see Navas \cite{NavasResolubles} and \S \ref{s-cyclic-lamplighter}. However, if $G$ is a finitely generated solvable subgroup of $\Diff^1_0([0, 1])$, its action on any minimal invariant subset $\Lambda\subset (0, 1)$ factors through a metabelian quotient.\end{rem}

Under the stronger assumption that actions be of class $C^2$, the above result was proved by Navas in \cite{NavasResolubles,NavasAmenable}, who also obtained an algebraic description of virtually  solvable subgroups of $\Diff^2_0([0, 1])$. Navas's approach is based essentially on the strong restriction on abelian groups of $C^2$ diffeomorphisms described above (in particular on Kopell's lemma). As also mentioned above, these obstructions for abelian subgroups fail in  $C^1$ regularity, and this approach does not even suffice to exclude that Plante-like actions of $\Z \wr \Z$ might be conjugate to any $C^1$ action. Our proof of Theorem \ref{mthm.C1} relies on  more global properties of the acting group. Roughly, it proceeds by first  reducing to the case of a Plante-like action (as a consequence of our analysis of general  $C^0$ actions of solvable groups, see Proposition \ref{p-plante-ubiquitous}), and then ruling out the possibility that the latter is semi-conjugate to any $C^1$ action (see  Proposition \ref{p-plante-C1}, which is inspired by another result from \cite{BMNR}). 

In the $C^0$ setting, our next result shows that the answer to Plante's problem is always affirmative in a \emph{local} sense,  namely  a semi-conjugacy to an affine action can be found at the level of germs near fixed points.

\begin{introthm}[Local semi-conjugacy near fixed points]\label{mthm.affine_trace}
	Let $\varphi\colon G\to \homeo_0(\R)$ be an irreducible  action of a finitely generated virtually solvable group.
	Then there exist an irreducible affine action $\psi\colon G\to \Aff(\R)$, an interval $I$ of the form $(a,+\infty)$, and a non-decreasing map $h\colon I\to \R$ with $\lim_{x\to +\infty} h(x)=+\infty$,  such that for every $g\in G$ we have
	\[\psi(g)(h(x))= h(\varphi(g)(x))\]
	for  any sufficiently large $x\in \R$. This action $\psi$ is unique up to affine conjugacy.
\end{introthm}

\begin{rem}
	Of course,  equivalent results hold for irreducible actions on finite intervals, considering neighborhoods of the endpoints.
\end{rem}

Theorem \ref{mthm.affine_trace} should be compared with a result of Chiswell and Kropholler \cite{ChiswellKropholler}, stating that a solvable group is left orderable if and only if it is locally indicable  (this is true more generally for amenable groups,  by a result of Witte Morris \cite{WitteMorris}). In an equivalent formulation, that result says that a finitely generated solvable group $G$ has a non-trivial action $\varphi\colon G \to \homeo_0(\R)$ if and only if it admits a non-trivial homomorphism to the group of translations $(\R, +)$. The reader may note that Theorem \ref{mthm.affine_trace} easily recovers this result, using that every non-trivial subgroup of $\Aff(\R)$ has a homomorphism to $(\R, +)$. An important difference is that Theorem~\ref{mthm.affine_trace} uniquely associates to every  action of $G$  on $\R$, a homomorphism to  $\Aff(\R)$ with an explicit dynamical meaning; in contrast, the proofs in \cite{ChiswellKropholler} and \cite{WitteMorris} proceed through considerations on the set of all left orders to deduce the existence of a homomorphism to $\R$, but do not  establish any explicit relation between a  fixed action (or order) and such homomorphisms.  

Our  third main result addresses the question of {rigidity} of the family of affine actions under small perturbations. To this end, given a group $G$, we denote by $\Homirr(G, \homeo_0(\R))$ the space of irreducible actions $\varphi\colon G\to \homeo_0(\R)$. Recall that this space has a natural topology, induced from the pointwise convergence topology of all maps $G\to \homeo_0(\R)$, with respect to the compact-open topology on $\homeo_0(\R)$. In view of Plante's problem, it is then natural to ask whether an action of a solvable group which is sufficiently close to an affine action must be semi-conjugate to an affine action. The following result provides an  affirmative answer for a vast class of solvable groups, namely those that are \emph{virtually metanilpotent}, that is, admitting a nilpotent normal subgroup $N$ such that $G/N$ is virtually nilpotent. This class includes in particular all virtually solvable \emph{linear} groups (that is, subgroups of $\mathsf{GL}(n, \mathbb{K})$ for some field $\mathbb{K}$), thanks to a well-known result of Mal'cev \cite{Malcev} (see Remark \ref{r-Malcev}).

\begin{introthm}[Perturbations of affine actions]\label{mthm.affine_rigid}
	Let $G$ be a finitely generated virtually metanilpotent group.
	Then, the subset of $\Homirr(G, \homeo_0(\R))$ of all irreducible actions which are semi-conjugate to a non-abelian affine action is open.
\end{introthm}

\begin{rem} Theorem \ref{mthm.affine_rigid} is false for general solvable groups: we construct
in \S \ref{ssc.ZwrZwrZ} a  finitely generated 3-step-solvable group $G$ and a sequence of irreducible actions  $(\varphi_n)\subset \Homirr(G, \homeo_0(\R))$ which are not semi-conjugate to any affine action, yet converge to a non-abelian affine action. 
\end{rem}
\begin{rem}
Theorem \ref{mthm.affine_rigid} implies in particular that if $\psi\colon G\to \Aff(\R)$ is a non-abelian affine action, then every irreducible action $\varphi$ sufficiently close to $\psi$ remains semi-conjugate to a non-abelian affine action. It is however  not true in general that ${\varphi}$ is semi-conjugate to $\psi$ itself: 
indeed even  metabelian groups often admit continuous paths of representations into $\Aff(\R)$ (see e.g.\ Proposition \ref{l-wr-affine}).
\end{rem}

\subsection{The structure theorem for general actions of solvable groups}

Although affine actions are a central character in the results stated above, the essence of this work is about understanding actions of solvable groups on the line that are \emph{not} semi-conjugate to any affine action. Our main structure theorem in this direction is Theorem \ref{t-main} below, which requires the framework of \emph{laminar actions} and the associated notion of \emph{horograding}, introduced and developed in our previous work \cite{BMRT}. We recall these notions and its main general properties in Section~\ref{sec generalities} (we have made an effort to make the presentation as self contained as possible).

To state Theorem \ref{t-main}, recall that an action $\varphi\colon G\to \homeo_0(\R)$ is {\em laminar} if it preserves a covering \emph{lamination} of the real line, that is, a closed subset $\mathcal{L}$ of the set of ordered pairs of points $\{(x, y) : x<y\}$ (thought of as a collection of bounded open intervals) which covers the line and is \emph{cross-free}, namely any two intervals $I, J\in \mathcal{L}$ are either nested or disjoint. The analogous notion of laminar actions on the circle is classical, and natural examples arise from actions of Fuchsian groups and 3-manifold groups, see for instance Calegari \cite{Calegaribook}, Baik \cite{LaminarBaik}, and references therein. As explained in \cite{BMRT}, laminar actions on the line are a source of somewhat ``exotic'' behaviors. For instance Plante's action of $\Z\wr\Z$ mentioned above, which not semi-conjugate to any affine action, is laminar (see Proposition \ref{prop Plante is laminar}). 

Although the fact of being laminar forces some restrictions on the action (for instance, every element has fixed points \cite[Lemma 8.1.9]{BMRT}),  to get stronger structural results we need to gain information on how $G$ acts on some invariant lamination. This is done through the notion of \emph{horograding} of a laminar action by another action of the group. The idea of this notion is to find an extra (in some sense transversal) direction for a group action on the line, where the group acts by another, hopefully simpler, action. To define it, observe that a lamination $\mathcal{L}$ of the line is naturally a partially ordered set $(\mathcal{L}, \subseteq)$ with respect to inclusion.

\begin{dfn}
Let $\varphi\colon G\to \homeo_0(\R)$ be a laminar action, and $j\colon G\to \homeo_0(\R)$ another irreducible action. A (positive) \emph{horograding} of $\varphi$ by $j$ is a pair $(\mathcal{L}, \hor)$ constisting of a $\varphi$-invariant lamination $\mathcal{L}$, and an order-preserving map $\hor \colon (\mathcal{L}, \subseteq)\to (\R, \le)$ such that $\hor(\varphi(g)(I))=j(g)(\hor(I))$ for every $I\in \mathcal{L}$ and $g\in G$.
\end{dfn}

We refer to Section \ref{sec generalities} and \cite[\S\S 8.2 and 15.1]{BMRT} for more details around this notion.
We already showed in  \cite[Theorem 8.3.8]{BMRT} that every minimal action of a finitely generated solvable group on the line is either conjugate to an affine action, or it is laminar; this observation was the starting point of this paper (for completeness, we streamline a short proof in Section \ref{sec Focal solvable}, as part of the proof of Theorem \ref{t-main}). The core result of this article shows that in the laminar case, the action can always be horograded by an action of strictly smaller complexity. 
For this, recall that the Fitting subgroup $\Fit(G)$ of a group $G$ is the subgroup generated by all normal nilpotent  subgroups of $G$. 

\begin{introthm} \label{t-main}
Let $\varphi\colon G\to \homeo_0(\R)$ be a minimal action of a finitely generated virtually solvable group. Then either
\begin{enumerate}
\item  $\varphi$ is conjugate to an affine action, or
\item \label{main-i-focal} $\varphi$ is laminar and can be horograded by an action $j\colon G\to \homeo_0(\R)$ which factors through $G/\Fit(G)$. Moreover this action $j$  is either minimal or cyclic.
\end{enumerate}
\end{introthm}

Note that  $G/\Fit(G)$ is a quotient of $G$ with strictly smaller (virtual) solvable length. This is really the crucial point in the statement and, as a consequence, Theorem \ref{t-main} can be applied inductively, by taking at each step the horograding action $j$ as the new $\varphi$: if the latter is not conjugate to any affine action, then it must be again laminar and  horograded by an action of a quotient of even smaller solvable length. In finitely many steps, we must reach an affine action $\psi\colon G\to \Aff(\R)$.  Thus Theorem \ref{t-main} shows that actions of a solvable group on the line can be naturally organized in a tower of finitely many levels of complexity, where  affine actions are the simplest and serve as a bedrock on which all actions are built.

Let us outline the connection between Theorem \ref{t-main} and the preceding results.
Theorem \ref{mthm.affine_trace} is the most direct application of Theorem \ref{t-main}: the action $\psi$ in the statement of Theorem \ref{mthm.affine_trace} is exactly the affine action obtained from $\varphi$ through the inductive procedure described above.  The proof of Theorem \ref{t-main} and \ref{mthm.affine_trace} are given in Section \ref{sec Focal solvable}.

When the group $G$ is virtually metanilpotent, the conclusion of Theorem \ref{t-main} is somehow stronger. Indeed, in this case we have that $G/\Fit(G)$ is virtually nilpotent, and by the result from Plante \cite{PlanteMeasure}, we know that every minimal action of  $G/\Fit(G)$ must be by translations.  So the horograding $\varphi'$ is an action by translations and the inductive procedure ends after one step. With this tool at hand, in Section \ref{s-metanilpotent} we  give the proof of Theorem \ref{mthm.affine_rigid}.

Finally, we use Theorem  \ref{t-main} to reduce the proof of Theorem \ref{mthm.C1} to the case of the group $G=\Z \wr \Z$. Indeed, we shall see in Section \ref{ssc.Plante} that Theorem \ref{t-main}  readily provides  a classification of actions of $\Z \wr \Z$ on the line up to semi-conjugacy: these are either affine or one of the four Plante-like actions described in \S \ref{sec ejemplo lamp lighter}. Using a similar argument, we show in Proposition \ref{p-plante-ubiquitous} that for every laminar action of a finitely generated solvable group  on the line, the group must have a subgroup isomorphic to $\Z \wr \Z$ whose action is semi-conjugate to a Plante-like action.  To finish the proof Theorem \ref{mthm.C1}, we show in Section \ref{sc.C1}  that no Plante-like action can be semi-conjugate to any $C^1$ action on the interval.

We point out that, in the proof Theorem \ref{t-main}, rather than working directly with invariant laminations, we use the dual point of view of group actions on  \emph{directed (real) trees}. Indeed for any laminar action $\varphi\colon G\to \homeo_0(\R)$, any $\varphi$-invariant lamination $(\mathcal{L}, \subseteq)$ of the real line can be completed into a partially ordered set $(\Tbb, \treeorder)$, which is isomorphic to a real tree endowed with a partial order induced by the choice of an end $\infty_\Tbb \in \partial \Tbb$, in such a way that the  $\varphi$-action on $\mathcal L$ extends to an action on $(\Tbb,\treeorder)$ (which fixes  the end $\infty_\Tbb$). To prove Theorem \ref{t-main}, we  show in \S \ref{subsection trees} that in this situation, any nilpotent normal subgroup $N$ of $G$ can be ``quotiented'' to obtain an action on a new  tree $\Tbb/N$ which ``branches less''.  Under the assumptions of Theorem \ref{t-main}, an inductive argument allows to reach the situation where $\Tbb$ has no branching at all, i.e.\ $\Tbb\cong \R$, and the group action on it is the desired horograding action.

In a follow-up paper, we will further use Theorem \ref{t-main} and the inductive approach developed here to provide more detailed results on the structure of semi-conjugacy classes in the space $\Homirr(G, \homeo_0(\R))$ for a finitely generated solvable group.

{\small \subsection*{Acknowledgments}
	The authors thank the anonymous referees for the several suggestions that improved the clarity of the presentation.
	
	All the authors acknowledge the support of the project MATH AMSUD, DGT -- Dynamical Group Theory (22-MATH-03), and the project ANR
	Gromeov (ANR-19-CE40-0007).
	J.B. was also partially supported with a “poste rouge” at ICJ (UMR CNRS 5208) by INSMI for IRL IFUMI members.
	N.M.B. was also partially supported by  the LABEX MILYON (ANR-10-LABX-0070) of Universit\'e de Lyon, within the program ``Investissements d'Avenir'' (ANR-11-IDEX-0007) operated by the French National Research Agency, and by IFUMI.
	C.R. is partially supported by FONDECYT 1210155 and FONDECYT 1241135.
	M.T. was also partially supported by
	the project ANER Agroupes (AAP 2019 Région Bourgogne--Franche--Comté), by CNRS with a
	semester of ``d\'el\'egation'' at IMJ-PRG (UMR CNRS 7586), and by the EIPHI Graduate School
	(ANR-17-EURE-0002).}

\section{Notation and preliminaries}
\subsection{Actions on the line}\label{sc.preliminary-actions}
Let us start with some notation and terminology. When a group $G$ acts on a set $X$, we simply write $g.x$ for the action, unless there is risk of confusion. If $\varphi\colon G\to \Bij(X)$ is an action, we write $\stab^\varphi(x)$ for the stabilizer of $x\in X$ in $G$ under the action $\varphi$,  $\fixphi(g)$ for the subset of fixed points in $X$ of an element $g\in G$, and $\suppphi(g)=X\setminus \fixphi(g)$ for its support.

Let $X$ be a real interval with non-empty interior (that  is, up to diffeomorphism,  one of the spaces $\R, [0,1], [0, 1)$). The interior of $X$ will be denoted by $\Int(X)$. The group of orientation-preserving homeomorphisms of $X$  is denoted by $\homeo_0(X)$, and similarly we denote the group of orientation-preserving $C^r$ diffeomorphisms $\Diff_0^r(X)$, for $r\ge 1$.  We will say that a group action $\varphi\colon G\to \homeo_0(X)$ is \emph{irreducible} if it has no fixed point in $\Int(X)$.
We also recall the notion of semi-conjugacy for actions on intervals (see Kim, Koberda, and Mj \cite{KKMj}).
\begin{dfn}
Let $X$ and $Y$ be real intervals with non-empty interior, and let $\varphi\colon G\to \homeo_0(X)$ and $\psi\colon G\to \homeo_0(Y)$ be two irreducible actions of a group $G$. We say that $\varphi$ and $\psi$ are \emph{semi-conjugate} if there exists a monotone map $h\colon \Int(X)\to \Int(Y)$ such that $h\circ\varphi(g)=\psi(g)\circ h$ for any $g\in G$. When $h$ is non-decreasing, we say that  $\varphi$ and $\psi$  are \emph{positively} semi-conjugate. When $h$ is an (orientation-preserving) homeomorphism, we say that $\varphi$ and $\psi$ are \emph{(positively) conjugate}.
\end{dfn} 

A well-known argument (see Navas \cite[Proposition 2.1.12]{Navas-book}) shows that when the group $G$ is finitely generated, every irreducible group action $\varphi\colon G\to \homeo_0(\R)$ admits a minimal non-empty closed  invariant subset  (that we shall call for short  a \emph{minimal invariant subset}), which is either $\R$, a perfect subset of empty interior, or a closed (and thus discrete) orbit. From the perspective of semi-conjugacy, this implies that $\varphi$ is always semi-conjugate to an action $\psi\colon G\to \homeo_0(\R)$ which  is either \emph{minimal} (that is, every orbit is dense), or which takes values in the group $(\Z, +)$ of integer translations, in which case we will say that $\psi$ is \emph{cyclic}. Moreover, such a minimal or cyclic representative of each semi-conjugacy class is unique up to conjugacy.  Thus when studying continuous actions of $G$ up to semi-conjugacy it is enough to restrict to minimal actions (the cyclic case being rather trivial), and we will often do so.

We will often build actions of groups on real intervals starting from actions on totally ordered sets. This strategy is quite general, but for the purpose of this note we will restrict our attention to actions on ordered sets $(\Omega,<)$ which are countable, have neither minimal nor maximal element, and which are \emph{densely  ordered} in the sense  that for every $\omega_1<\omega_2$ in $\Omega$ there is $\omega\in \Omega$ such that $\omega_1<\omega<\omega_2$. The good thing about these sets is that Cantor's ``back and forth argument'' ensures that there is an order-preserving bijection $t\colon  (\Omega,<)\to (\Q,<)$ where rationals are equipped with the natural order (see for instance Clay and Rolfsen \cite[Theorem  2.22]{ClayRolfsen}).

\begin{dfn} \label{def dynamical realization}
	Let  $(\Omega,<)$ be a countable densely ordered set having neither maximal nor minimal element, $t\colon \Omega\to \Q$ an order-preserving bijection, and $\varphi\colon G\to  \Aut(\Omega,<)$  a non-trivial order-preserving action. A \emph{dynamical realization} of $\varphi$ is an irreducible action $\bar\varphi\colon G\to \homeo_0(\R)$ with the property that
	$\bar\varphi(g)(t(\omega))= t(\varphi(g)(\omega))$
	for every $g\in G$ and every $\omega\in \Omega$.
\end{dfn}
A dynamical realization always exists, see \cite[\S 2.4]{ClayRolfsen} for details. Moreover, if $\bar\varphi_1$ and $\bar\varphi_2$ are two dynamical realizations of $\varphi\colon G\to \Aut(\Omega,<)$, then they are positively conjugate. Indeed, if $t_1$ and $t_2$ are the corresponding bijections from $\Omega$ to $\Q$, then $h=t_1\circ t_2^{-1}$ can be extended to the whole real line so that it conjugates $\bar\varphi_1$ to $\bar\varphi_2$. This is why we usually refer to \emph{the} dynamical realization of $\varphi$. The following result will be useful later.

\begin{lem} \label{l-dynamical-realisation}
Let $(\Omega,<)$ be a countable densely ordered set with neither maximal nor minimal element, and $\varphi\colon G\to \Aut(\Omega,<)$ an order-preserving action whose dynamical realization $\bar\varphi\colon G\to \homeo_0(\R)$ is minimal.
Let $\psi\colon G\to \homeo_0(\R)$ be an irreducible action, and assume that   $\sigma\colon \Omega\to \R$  is a map which is non-decreasing and $G$-equivariant, in the sense that
\[ \sigma (\varphi(g)(\omega))=\psi(g)(\sigma(\omega)) \quad \quad \text{for every }g\in G, \omega\in \Omega.\]
Then, $\sigma$ is injective, and  $\psi$ and $\bar{\varphi}$ are positively semi-conjugate. In particular, when $\psi$ is minimal, it is positively conjugate to $\bar{\varphi}$. 
\end{lem}

\begin{proof}
Let $t\colon \Omega\to \Q$ be the bijection used to define the dynamical realization $\bar\varphi$. We define a map $\tilde{\sigma}\colon \R \to \R$ by setting $\tilde{\sigma}(x)=\sup \{\sigma(\omega) : t(\omega) \le x\}$. Then $\tilde{\sigma}$ is non-decreasing and equivariant with respect to the $\bar{\varphi}$-action on the source and the $\psi$-action on the target, showing that $\psi$ is positively semi-conjugate to $\bar{\varphi}$. We claim that $\tilde{\sigma}$ is injective. To see this, consider the subset 	
	\[S=\{x\in\R:\tilde\sigma\text{ is constant on a neighborhood of }x\}.\]
It is direct to see that $S$ is open and $\bar\varphi$-invariant, and therefore empty by minimality of $\bar{\varphi}$. To conclude the proof of the claim, notice that if $\tilde{\sigma}(x)\neq\tilde\sigma(y)$ for $x<y$ then, monotonicity of $\tilde\sigma$ would imply $(x,y)\subseteq S$, which is a contradiction. In  particular, since $t$ is injective, so is $\sigma=\tilde{\sigma}\circ t$.
	
Assume now that $\psi$ is minimal. In this case, the image of $\tilde\sigma$ must be dense, or otherwise its closure would be a proper closed $\psi$-invariant subset. Therefore, since $\tilde{\sigma}$ is monotone, injective, and has dense image, it must be a homeomorphism. This shows that $\bar{\varphi}$ and $\psi$ are positively conjugate.
\end{proof}

\subsection{Virtually solvable groups and Fitting series} \label{s-fitting}
In this paper, we will be mostly interested in groups that are  \emph{virtually solvable}, that is, contain a (normal)  solvable subgroup of finite index.  We refer the reader to the book by Lennox and Robinson \cite{LDR} for a general reference on the theory of solvable groups. 

Recall that the \emph{Fitting subgroup} $\Fit(G)$ of a group $G$ is the subgroup generated by all nilpotent normal subgroups of $G$. Fitting's theorem states that the group generated by two nilpotent normal subgroups of $G$ is again nilpotent. In particular, $\Fit(G)$ is in fact the union of all nilpotent normal subgroups. The following basic lemma will be used without mention. 
\begin{lem} \label{l-Fitting}
Let $G$ be a group, and $G_0\unlhd G$ a normal solvable subgroup of finite index. Then  $\Fit(G)$  contains every nilpotent normal subgroup of $G_0$.
\end{lem}
\begin{proof}
Let $N\unlhd G_0$ be nilpotent. Since $G_0$ is normal and of finite index in $G$, the subgroup $N$ has finitely many conjugate subgroups in $G$ and all of them are contained in $G_0$. By Fitting's theorem, the subgroup $M$ generated by such conjugates is nilpotent, and it is normal in $G$. Hence $N\le M\le  \Fit(G)$. 
\end{proof}

Note that in the previous situation $\Fit(G)$ contains in particular the last non-trivial term of the derived series of $G_0$.
\begin{dfn}
Let $G$ be a virtually solvable group. The \emph{upper Fitting series} of $G$ is defined by setting $F_1=\Fit(G)$, and $F_{i+1}$ to be the preimage of $\Fit(G/F_i)$ in $G$. 
Since $G$ is virtually solvable, there must exist some $k\in \N$ such that $G/F_k$ is finite. We call the smallest such $k$ the \emph{virtual Fitting length} of $G$ and denote it by $\vf(G)$. 
\end{dfn}

\subsection{The group $\Z\wr\Z$ as an illustrative example} \label{sec ejemplo lamp lighter}
Throughout the paper, we will frequently use the example of the group $\Z \wr \Z$. To get the reader immediately familiar with this example, in this subsection we describe three natural families of actions of the group $\Z\wr\Z$ on the real line by orientation-preserving homeomorphisms. Actions in the first two families are semi-conjugate to affine actions, while actions in the third family are what we call \emph{Plante-like} actions,  described in  \S \ref{ssc.plante}.  We will see later in Proposition \ref{p-ZwrZ} that, up to semi-conjugacy, all the possible actions of $\Z \wr \Z$ on the line arise through the constructions presented here.  

Throughout this subsection we set $\Z \wr \Z=L\rtimes \Z$, where $L=\bigoplus_\Z \Z$ is the group of finitely supported configurations (functions) $f\colon \Z\to \Z$, and $\Z$ acts on $L$ by shifting the indices. We fix the generating set $\{g,h_0\}$ of $\Z\wr\Z$, where $g$ generates the factor $\Z$, and $h_0\in L$ is the configuration given by $h_0(0)=1$ and $h_0(k)=0$ for $k\neq 0$, and denote by $h_n=g^nh_0g^{-n}$ ($n\in \Z$) the elements given by $h_n(n)=1$ and $h_n(k)=0$ for $k\neq n$.

The group $\Z\wr\Z$ is solvable (more precisely, metabelian) and it is clear that $L$ agrees with $\Fit(\Z\wr\Z)$, the Fitting subgroup of $\Z\wr\Z$. Moreover, $\Z\wr\Z$ admits the  infinite presentation
\begin{equation}\Z\wr\Z\cong\langle g, h_0 \mid h_n=g^nh_0g^{-n}, \, [h_n, h_m]= \id \quad (n, m\in \Z)\rangle. \label{e-presentation-ZwrZ}\end{equation}

\subsubsection{Faithful actions semi-conjugate to cyclic ones} \label{s-cyclic-lamplighter}
Perhaps the simplest way of producing faithful actions  $\varphi\colon\Z\wr\Z\to \homeo_0(\R)$ is to let $\varphi(g)$ be a homeomorphism without fixed points, and $\varphi(h_0)$ any non-trivial homeomorphism whose support is contained in an open interval $D$ such that  $\varphi(g)(D)\cap D\neq \varnothing$\footnote{As observed by Navas in \cite{NavasResolubles},  by iterating this method one can produce $C^\infty$ actions of the groups $H_{n+1}:=H_n\wr \Z$ on the line, recursively defined after setting $H_1=\Z$. Note that $H_n$ is solvable of derived length $n$.}. In this way, $\varphi(h_n)$ is supported on $\varphi(g^n)(D)$ and hence the $\varphi(h_n)$ pairwise commute, generating an isomorphic copy of $L$, and we have $\langle \varphi(g),\varphi(h_0)\rangle\cong \Z\wr\Z$. These actions, however, are all semi-conjugate to a cyclic action (in particular they are semi-conjugate to an affine action). This is a consequence of the fact that every such action admits  a closed invariant subset, namely the complement of $\bigcup_n \varphi(g^n)(D)$,  on which $\varphi(h_0)$ acts trivially (so the action on it  is generated by $\varphi(g)$). 
For instance, in the case $\varphi(g)(x)=x+1$ and $D=(0,1)$, it is straightforward to check that  the ``closest lower integer'' function $h\colon x\mapsto\max \{n\in \Z:n\le x\}$ semi-conjugates $\varphi$ to the cyclic action $\psi$ given by  $\psi(g)\colon x\mapsto x+1$ and $\psi(h_0)=\id$.

\subsubsection{Affine actions} 
Another way of producing actions of $\Z\wr\Z$ on the line is by considering affine maps
\[\varphi(g) \colon x\mapsto \lambda x+\beta , \quad \varphi(h_0)\colon x\mapsto x+  \alpha  \quad \quad (\lambda>0\text{ and } \alpha, \beta \in\R).\]
The reader can check that for any choice of $\lambda,\alpha,\beta$ as above, the images of the generators satisfy the relations from  \eqref{e-presentation-ZwrZ},   and hence the action $\varphi\colon\Z\wr\Z\to \homeo_0(\R)$ is well defined. These affine actions can be readily classified up to conjugacy.

\begin{prop} \label{l-wr-affine} Let $\{g,h_0\}$ be the generating set for $\Z\wr\Z$ as above. Then every irreducible affine action $\varphi\colon \Z \wr \Z\to \Aff(\R)$ is positively conjugate by an affine map to an action in one of the following families:
\begin{itemize}
\item  actions by translations obtained by setting 
\[\varphi(h_0)\colon x\mapsto x+\alpha, \quad \varphi(g)\colon x\mapsto x +\beta \quad \quad  (\alpha^2+\beta^2=1);\]
\item non-abelian affine actions obtained by setting
\[\varphi(h_0)\colon x\mapsto x\pm 1, \quad \varphi(g) \colon x\mapsto \lambda x\quad \quad (\lambda>0).\]
\end{itemize}
\end{prop}
\begin{proof}
Let $\varphi$ be an irreducible affine action of $\Z\wr\Z$. First we show that $\varphi(h_0)$ is a translation. Assume by contradiction that $\varphi(h_0)$ is a homothety, with fixed point $p$. Then $p$ cannot be fixed by $\varphi(g)$ or it would be a global fixed point. Thus $\varphi(h_1)$ is a homothety whose fixed point is $\varphi(g)(p)\neq p$. This contradicts that $\varphi(h_0)$ and $\varphi(h_1)$  commute. Thus $\varphi(h_0)$ is a translation (possibly trivial). If $\varphi(g)$ is also a translation, up to conjugation by an affine map we are in the first case of the classification. If $\varphi(g)$ is a homothety, then $\varphi(h_0)$ cannot be trivial, and upon conjugating by an affine map we can assume that $\varphi(h)$ is a translation by $1$ and that $\varphi(g)$ fixes 0, which yields the second case of the classification.
\end{proof}

\begin{rem} 
In the previous classification, note that  non-abelian affine actions are faithful if and only if  $\lambda$ is a transcendental number (since the translations $\varphi(h_n)(x)=x+\lambda^n$ generate a free abelian group precisely when $\lambda $ is transcendental). Actions by translations obviously descend to actions of the abelianization of $\Z \wr \Z$, which is isomorphic to $\Z^2$. The abelianization acts faithfully provided $\alpha$ and $\beta$ are rationally independent. 
\end{rem}

 \subsubsection{Plante-like actions} 
\label{ssc.plante}  Finally, we describe a procedure for producing actions of $\Z\wr\Z$ on the line that are not semi-conjugate to any (real) affine action. Up to positive conjugacy, there will be four such actions (two, if we also allow negative conjugacies) which are what we call  \emph{Plante-like actions} of $\Z\wr\Z$. One of these actions was first considered by Plante in \cite{Plante}, who constructed it by explicitly defining two homeomorphisms of the line. Following the discussion in our previous work  \cite[Example 8.1.8]{BMRT}, we shall take a different point of view, and obtain such actions as the dynamical realization of an affine action of $\Z\wr\Z$ on a countable ordered set (this point of view can be easily generalized to other wreath products of groups, as explained below).

Recall that we write $\Z \wr \Z=L\rtimes \Z=\langle g,h_0\rangle$, where $L=\bigoplus_\Z \Z$ and $\{ g, h_0\}$ from \eqref{e-presentation-ZwrZ} is our preferred generating set, and that we set $h_n=g^n h_0 g^{-n}$. 

One can identify the subgroup $L$ with the ring of Laurent polynomials $\Z[X,X^{-1}]$, choosing the identification so that $h_n\in L$  represents the polynomial $ X^n$ (in particular, $h_0$ represents the constant polynomial $1$). With this point of view, the group $\Z\wr\Z$ admits a natural action on $\Z[X,X^{-1}]$, where $L\cong \Z[X,X^{-1}]$ acts additively and the generator $g$ acts by multiplication by $X$. This action of $\Z \wr \Z$ on $\Z[X,X^{-1}]$ preserves four natural orders $\prec_{\max}^\pm$ and $\prec_{\min}^\pm$ of lexicographic type, obtained by looking at the sign of the coefficient of the monomial of maximal or minimal degree, respectively. More precisely, given a non-zero Laurent polynomial
$P(X)$,  we declare $0\prec_{\max}^+ P(X)$  if the coefficient of the highest power of $X$ in $P(X)$ is a non-negative integer, and $0\prec_{\max}^- P(X)$ if the coefficient of the highest power of $X$ in $P(X)$ is a non-positive integer.    Similarly, we declare $0\prec_{\min}^+ P(X)$ (respectively, $0\prec_{\min}^- P(X)$) if the coefficient of the lowest power of $X$ in $P(X)$ is a non-negative integer (respectively, non-positive integer). It is routine to check that these four orders are invariant under the action of $\Z \wr \Z$ defined above. By considering the dynamical realizations of the above actions, we obtain, up to positive conjugacy, four actions on the line which are what we call the \emph{Plante-like actions} of $\Z \wr\Z$.

Let us take a closer look at Plante-like actions. Let $\prec$ be one of the orders $\prec_{\max}^\pm, \prec_{\min}^\pm$ on $\Z[X, X^{-1}]$, and $\varphi\colon G\to \homeo_0(\R)$ the corresponding Plante-like action. On the one hand, by definition of the lexicographic order $\prec$, for every $h\in L$ we have that every $h$-orbit in $(\Z[X,X^{-1}], \prec)$ is bounded. For example, in the case of $\prec_{\max}^+$, the $h$-orbit of any $Q\in \Z[X, X^{-1}]$ is contained in the interval $(-X^m, X^m)$ whenever $m$ is any integer strictly larger than the degree of $Q$ and of the polynomial representing $h$.  However, the global action of the group $L$ on $\Z[X, X^{-1}]$ is transitive. This implies that for the Plante-like action $\varphi$, the group $L$ acts without fixed points (in particular, $\varphi$ is irreducible), while every $h\in L$ has infinitely many fixed points. This immediately prevents the action $\varphi$ to be semi-conjugate to any affine action (cf.\  Proposition \ref{l-wr-affine}).

Consider now the action of the generator $g$ on $(\Z[X, X^{-1}], \prec)$.  Since $g$ acts by multiplication by $X$,  it has exactly one fixed point (the polynomial $0$). Suppose, for definiteness, that $\prec$ is one of the orders $\prec_{\max}^\pm$. Then  for any  Laurent polynomial  $P(X)\succ 0$, we have  $g\cdot P(X)=XP(X)\succ P(X)$, and moreover $g^n\cdot P(X)=X^nP(X)$ is arbitrarily large as $n\to +\infty$, and arbitrarily close to $0$ as $n\to -\infty$ (this easily follows from the definition of the lexicographic order, since applying $g^n$ to $P(X)$ shifts the largest power of $X$ by $n$, without changing the corresponding coefficient). Similarly for $P(X)\prec 0$, we have that $g^n\cdot P(X)$ gets arbitirarily small as $n\to +\infty$, and arbitrarily close to $0$ as $n\to -\infty$. We deduce that in the Plante-like action, $\varphi(g)$ is an expanding homothety (in the sense of Definition \ref{d-expanding-homothety}). Similarly one shows that $\varphi(g)$ is a contracting homothety in the case where $\prec$ is $\prec_{\min}^\pm$. We can (and will) assume that the unique fixed point of $\varphi(g)$ is the origin of the real line. In Figure  \ref{fig Plante} the Plante-like action corresponding to $\prec_{\max}^+$ is depicted.

\begin{figure}[ht]
     \includegraphics[width=0.5\textwidth]{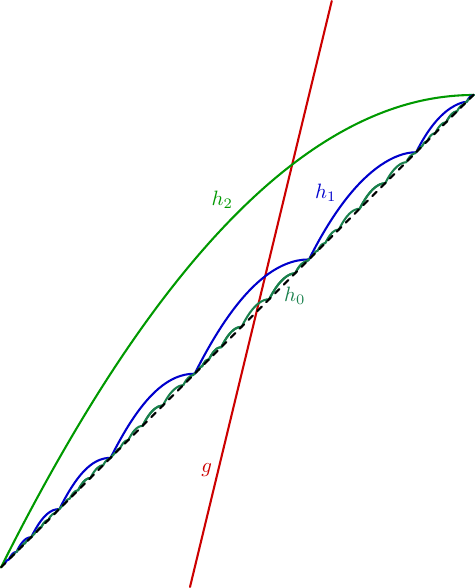}
      \caption{ Plante-like action corresponding to the order $\prec_{\max}^+$}
      \label{fig Plante}
\end{figure}

After these considerations, we can prove the following. 

\begin{prop} A Plante-like action of $\Z\wr\Z$  is minimal. 
\label{prop Plante-like minimal}
\end{prop}

\begin{proof} Let $x\in \R$ be a point and $I\subset \R$ an open interval. We need to check that the orbit of $x$ meets the interval $I$.
Since $I$ is open, it follows that it contains rational points and hence points in the image of the order-preserving bijection $t\colon \Z[X,X^{-1}]\to \Q$  from which the dynamical realization is built (see Definition \ref{def dynamical realization} and the discussion before it). In particular, since $L$ acts transitively on the image of $t$, there is $h\in L$ such that $\varphi(h)(I)$ contains $0\in \R$. Now, recall that we are assuming that the origin $0$ is the unique fixed point of $\varphi(g)$, which moreover acts as a homothety (expanding or contracting, depending on the order considered on $\Z[X,X^{-1}]$). It follows that there is $n\in \Z$ such that $\varphi(g^nh)(I)$ contains the point $x$, and hence $\varphi(h^{-1}g^{-n})(x)$ belongs to $I$, as desired.
\end{proof}

\section{Generalities on laminar actions and horogradings} 
\label{sec generalities}

In this section we recall general facts on laminar actions and horogradings. Large parts of the section are an exposition of the concepts in our previous work \cite[Chapters 8 and 11]{BMRT} (although we sometimes provide self-contained proofs for  the reader's convenience). The main new ingredients are the existence of minimal laminations for laminar actions of finitely generated groups (Proposition \ref{prop.minimalamination}), and the corresponding concept in terms of actions on directed trees (Proposition \ref{prop.minimalmodeltree}).

\subsection{Laminations and horogradings}

We say that two open intervals $I,J\subset \R$  \emph{do not cross} if they are either nested or disjoint, that is either $I\subset J$, or $J\subset I$, or $I\cap J=\varnothing$.  A \emph{prelamination} of the real line is a collection $\mathcal{L}$ of (non-empty) bounded open  intervals, that pairwise do not cross. We shall identify the set of all bounded open intervals in $\R$  with the space $\R^{(2)}:=\{(x, y)\in \R^2: x<y\}$, endowed with the topology induced from $\R^2$ (i.e.\ convergence of endpoints). A \emph{lamination} is a prelamination which is moreover a closed subset of $\R^{(2)}$. A lamination is \emph{covering} if it defines a cover of $\R$; equivalently, if it contains an increasing exhaustion of $\R$. Elements of a (pre)laminations will be often called \emph{leaves}.
 
Note that a prelamination $\mathcal{L}$ is naturally a partially ordered set (\emph{poset}, for short) with respect to the inclusion order $\subseteq$. Given $l_0, l_1\in \mathcal{L}$, we will say that $l_1$ is the successor of $l_0$ if it is the smallest element of $(\mathcal{L}, \subseteq)$ such that $l_0\subsetneq l_1$.

\begin{rem}
	If an action $\varphi \colon G \to \homeo_0(\R)$ has an invariant prelamination $\mathcal{L}$, then the closure $\overline{\mathcal L}$ in $\R^{(2)}$ is an invariant lamination.  It is convenient to visualize a prelamination of the line as a collection of pairwise disjoint semi-circles in the upper half-plane (i.e.\ geodesic lines in the Poincaré metric), whose endpoints define the leaves of the prelamination.
\end{rem}

\begin{dfn}\label{def.lamination}
An irreducible action $\varphi\colon G\to\homeo_0(\R)$ is \emph{laminar} if it preserves a covering prelamination $\mathcal{L}$.  We also say that the laminar action $\varphi$ is \emph{focal} if, in addition, there exists $l\in\mathcal{L}$  and a sequence $(g_n)\subset G$ such that  $(\varphi(g_n)(l))$ is an  increasing  exhaustion of the line.  When we want to emphasize the role of the prelamination, we will say that $\varphi$ is laminar (and focal) \emph{with respect to $\Lcal$}.
\end{dfn}

\begin{rem} 
 The property of being a laminar (focal) action is invariant under semi-conjugacy, and every focal laminar action is semi-conjugate to a \emph{minimal} focal laminar action \cite[Proposition 8.1.14]{BMRT}.
\end{rem}
\begin{rem}\label{r-minimal-implies-focal}
Conversely, assume that $\varphi\colon G\to\homeo_0(\R)$ is a {minimal} action. If $\mathcal{L}$ is any  non-empty $\varphi$-invariant prelamination, then $\mathcal{L}$ is automatically covering, and $\varphi$ is focal with respect to $\mathcal{L}$ \cite[Proposition 8.1.15]{BMRT}.
\end{rem}

Laminar actions satisfy some general constraints. In particular group elements can be classified according to two types of dynamics.  For this, we recall the following terminology. 

\begin{dfn} \label{d-expanding-homothety}
	A homeomorphism $h\in\homeo_0(\R)$ is an \emph{expanding pseudo-homothety} if there exists a compact subset $K\subset \R$ such that $h(U)\supset \overline{U}$ for every open subset $U\supset K$.
	When $h^{-1}$ is an expanding pseudo-homothety, we say that $h$ is a \emph{contracting pseudo-homothety}. In the case where $K$ can be taken to be a point, we simply say $h$ is an (expanding or contracting) \emph{homothety}. We also say that $h\in \homeo_0(\R)$ is \emph{totally bounded} if $\Fix(h)$ accumulates on both $+\infty$ and $-\infty$.
\end{dfn}
It is shown in \cite[Proposition 8.1.13]{BMRT} that, for a laminar action $\varphi \colon G\to \homeo_0(\R)$, every element in $\varphi(G)$ is either a pseudo-homothety or  totally bounded.

We now recall the notion of horograding from \cite{BMRT}.
Let $\mathcal{L}$ be a prelamination of the real line.  A   \emph{horograding} of $\mathcal{L}$ is a monotone map $\hor \colon (\mathcal{L}, \subseteq) \to (\R, \leq )$, namely a map such that $\hor(l_1)\leq \hor(l_2)$ whenever $l_1\subseteq l_2$ (in this case we say that $\hor$ is a \emph{positive} horograding) or $\hor(l_1)\geq \hor(l_2)$ whenever $l_1\subseteq l_2$ (in this case we say that $\hor$ is \emph{negative}). In the presence of a group action, this leads to the following notion.

\begin{dfn}
Let $\varphi\colon G\to \homeo_0(\R)$ be a laminar action, and $j\colon G\to \homeo_0(\R)$ an irreducible action. A (positive or negative) \emph{horograding of $\varphi$ by $j$} is a pair $(\mathcal{L}, \hor)$ consisting of a $\varphi$-invariant covering lamination, and a  (positive or negative, accordingly) horograding $\hor\colon (\mathcal{L}, \subseteq) \to \R$ such that $\hor(\varphi(g)(l))=j(g)(\hor(l))$, for every $g\in G$ and $l\in \mathcal{L}$.
\end{dfn}

\begin{rem}[Lemma 8.2.7 in \cite{BMRT}] \label{r-horograding-proper}
Suppose that $(\mathcal{L}, \hor)$ is a positive horograding of a laminar action $\varphi\colon G\to \homeo_0(\R)$ by an irreducible action $j\colon G\to \homeo_0(\R)$. Then for every increasing sequence $(l_n)\subset \mathcal{L}$ that exhausts $\R$,  we must have $\hor(l_n)\to +\infty$. Indeed, since $j$ is irreducible and $\hor(\mathcal{L})$ is $j$-invariant, for every $M>0$ there exists $l\in \mathcal{L}$ such that $\hor(l)\ge M$, and since  $l_n$ eventually contains any given $l$, the conclusion follows. 
\end{rem}

The existence of a horograding of a laminar action $\varphi$ by an action $j$ implies that the action $j$ retains information on the large-scale behavior of $\varphi$. In particular, the type of each element with respect to $\varphi$ is determined by $j$ as follows.

\begin{prop}[Proposition 8.2.10 in \cite{BMRT}]\label{p-classification-elements} Let $\varphi\colon G\to\homeo_0(\R)$ be a minimal laminar action positively horograded by $j\colon G\to\homeo_0(\R)$. For any element $g\in G$, we have the following.
\begin{enumerate}
\item \label{i-pseudo-homothety} If $\Fix(j(g))$ does not accumulate on $+\infty$, then $\varphi(g)$ is a pseudo-homothety, which is expanding if $j(g)(x)>x$ for any sufficiently large $x$, and contracting otherwise. 
\item Otherwise, $\varphi(g)$ is totally bounded.
\end{enumerate}
	Moreover, in the former situation, if $\Fix(j(g))=\emptyset$ then $\varphi(g)$ is a homothety. 
\end{prop}

The following result establishes a more explicit relation between a laminar action and an action that horogrades it. Roughly speaking, it says that if $\varphi \colon G\to \homeo_0(\R)$ is horograded by $j\colon G\to \homeo_0(\R)$,  then there are locally defined semi-conjugacies at the level of germs near $\infty$. This statement is a simplified version of the discussion in \cite[\S 15.1]{BMRT} (which gives a more explicit conclusion assuming $G$ finitely generated); for the reader's convenience we give a self-contained proof.

\begin{prop}\label{p-horograding-germ}
Let $\varphi \colon G\to \homeo_0(\R)$ be a minimal laminar action, positively horograded by $j\colon G\to \homeo_0(\R)$. 
Then, there exist two maps $h_+, h_-\colon \R\to \R$ satisfying the following conditions:
\begin{enumerate}
	\item $h_+$ is monotone non-decreasing,
	\item $h_-$ is monotone non-increasing,
	\item \label{i-maps-proper} we have $\lim_{x\to +\infty}h_+(x)=\lim_{x\to -\infty} h_-(x)=+\infty$,
	\item \label{i-partial-equivariance} for every $g
	\in G$ there exists $M>0$ such that $ h_+(\varphi(g)(x))=j(g) ( h_+(x))$ for every $x>M$, and $ h_-(\varphi(g)(x))=j(g)( h_-(x))$ for every $x<-M$. 
\end{enumerate}
\end{prop}

\begin{proof}  Let $(\mathcal{L}, \hor)$ be a positive horograding of $\varphi$ by $j$.  Fix a leaf ${k}:=(a, b)\in \mathcal{L}$, and set $\mathcal{F}=\{l\in \mathcal{L}: l\supseteq k\}$.  Note that $\mathcal{F}$ is a totally ordered closed subset of $\mathcal{L}$ (with respect to the inclusion order $\subset$), so for any point $x\notin {k}$ we can define
 \[l_x=\min\{l\in \mathcal{F}: x\in \overline{l}\}.\] 
 The map $x\in \R\mapsto l_x\in (\mathcal{L}, \subset)$ is non-decreasing on $[b, +\infty)$ and  non-increasing on  $(-\infty, a]$. (For example, suppose that $x < y$ are in $[b, +\infty)$. Then the leaf $l_y$ contains both the leaf $k=(a, b)$ and the point $y$, thus it actually contains the interval $(a, y)$; in particular $x\in l_y$, so $l_x\subseteq l_y$).   By postcomposing it with the horograding $\hor$, we obtain a non-increasing map $h_+\colon [b, +\infty)\to \R$ and a non-decreasing map $h_-\colon (-\infty, a]\to \R$, given by  $h_\pm(x)= \hor(l_x)$.
 To have maps defined on the whole line (as in the statement), we can extend each map $h_\pm$ arbitrarily to a monotone map $h_\pm \colon \R \to \R$  (this is possible since the image of both maps is bounded below by $\hor({k})$).  As $x\to \pm \infty$, the leaves $l_x$ are arbitrarily large in $(\mathcal{L}, \subseteq)$, so the maps $h_\pm$ satisfy \eqref{i-maps-proper}, see Remark \ref{r-horograding-proper}.
 
We now proceed to show that the maps $h_\pm$ satisfy \eqref{i-partial-equivariance}. We will write here $g.x$ for $\varphi(g)(x)$, and similar shorthand notation, for sake of readability.
Fix $g\in G$. Since $\mathcal{L}$ is a closed subset of $\R^{(2)}$, the subset of leaves that contain ${k}\cup g^{-1}.k$ has a smallest element. Denote this leaf by $u_g$. 

 \begin{claim}
 For $x\notin \overline{u}_g$, the points $x$ and $g.x$ belong to the same connected component of $\R\setminus \overline{k}$, and we have $g.l_x=l_{g.x}$.
 \end{claim}
 \begin{proof}[Proof of claim]
The assumption $x\notin \overline{u}_g\supset k$ implies in particular $x\notin k$, and since $g.\overline{u}_g\supset {k}$, the point $g.x$ must belong to the same connected component of $\R\setminus \overline{k}$ as $x$. In particular both leaves $l_x$ and   $l_{g.x}$ are well defined. Now, on the one hand $g.l_x$ contains $g.x$ in its closure, and also contains $g.u_g\supseteq k$, so $g.l_x\supseteq l_{g.x}$. On the other hand, the leaves $l_{g.x}$ and $g.u_g$ intersect non-trivially (as both contain $k$), so must be related by inclusion. Since $g.x$ is in the closure of $l_{g.x}$, but not in the closure of $g.u_g$, the only possible inclusion between them is $g.u_g\subseteq l_{g.x}$. Hence $g^{-1}.l_{g.x}$ is a leaf containing $u_g\supseteq k$, and also containing $x$ in its closure, so $g^{-1}.l_{g.x}\supseteq l_{x}$, i.e.\ $l_{g.x}\supseteq g.l_x$. We have shown that $g.l_x=l_{g.x}$.
\end{proof}
 
From the claim we deduce that, if $x\notin \overline{u}_g$, then 
 \[h_\pm (g.x) =\hor(l_{g.x})=\hor(g.l_x)=j(g)(\hor(l_x))=j(g)(h_{\pm}(x)),\]
showing \eqref{i-partial-equivariance}.
\end{proof}

\subsection{The example of Plante-like actions} \label{s-plante-laminations}
\subsubsection{Plante-like actions of $\Z \wr \Z$}
Here we explain why the Plante-like actions of $\Z \wr \Z$ (see \S \ref{ssc.plante}) are laminar and horograded by a cyclic action. 

\begin{prop}\label{prop Plante is laminar}
	A Plante-like action $\varphi\colon \Z \wr \Z \to \homeo_0(\R)$ is laminar and horograded by a cyclic action.
\end{prop}

\begin{proof}
We resume notation from \S \ref{sec ejemplo lamp lighter}. In particular, recall that we denote by $\{g, h_0\}$ the standard generating pair of $\Z \wr \Z$,  satisfying the presentation  \eqref{e-presentation-ZwrZ}, and set $h_n:=g^n h_0 g^{-n}$.
For $n\in \Z$, we let $\mathcal{L}_n$ be the set of connected components of $\suppphi(h_n)$, and $\mathcal{L}=\bigcup_n \mathcal{L}_n$. For $l\in \mathcal{L}_n$, write $\hor(l):=n$. We want to prove that $\mathcal{L}$ is a $\varphi$-invariant lamination (if so, since $\varphi$ is minimal \ref{prop Plante-like minimal}, Remark \ref{r-minimal-implies-focal} guarantees that it is laminar), and that $(\mathcal{L}, \hor)$ is a horograding of $\varphi$ by the cyclic action $j\colon \Z \wr \Z\to \homeo_0(\R)$, given by $j(g):x\mapsto x+1$ and $j(h_0)=\id$.

We keep denoting by $L=\bigoplus_\Z \Z\cong \Z[X, X^{-1}]$  the group of lamps. 
Recall from \S \ref{ssc.plante} that there are four Plante-like actions of $\Z \wr \Z$, associated with four lexicographic orders on $\Z[X, X^{-1}]$. 
We only discuss the case of the order  $\prec_{\max}^+$ (the other cases are analogous), and start with the following observation (that we isolate for further reference).

\begin{claimnum} \label{l-plante-fixed-points}
For every $n\in \Z$, we have $\suppphi(h_n)\subsetneq \suppphi(h_{n+1})$. More precisely, every connected component of $\suppphi(h_n)$ is compactly contained in $\suppphi(h_{n+1})$.
\end{claimnum}
\setcounter{claimnum}{0}
\begin{proof}[Proof of claim]
Since each $h_n$  acts without fixed points on $\Z[X,X^{-1}]$, the open subset $\suppphi(h_n)$  contains $\Q$, which is the image of the order-preserving bijection $t\colon \Z[X,X^{-1}]\to \Q$ used to define the dynamical realization $\varphi$. More precisely, if for every $P\in \Z[X, X^{-1}]$ we let $I(n, P)$ denote the connected component of $\suppphi(h_n)$ containing $t(P)$, we have that every connected component of $\suppphi(h_n)$ is of the form $I(n, P)$, by density of $\Q$. Note that each $I(n, P)$ is equal to the convex hull of  $\{h_n^k.t(P): k\in \Z\}$.

Recall that $h_n$ corresponds to the polynomial $X^n$ under the identification  $L\cong \Z[X, X^{-1}]$. Using this identification, for $f_1, f_2\in L$, we have that $f_1.t(P)< f_2.t(P)$ if and only if $f_1+P\preceq_{\max}^+ f_2+P$, which in turn is equivalent to $f_1\preceq_{\max}^+ f_2$. Now, by the definition of the order $\preceq_{\max}^+$, we have that
\[-X^{n+1} \preceq_{\max}^+ kX^n \preceq_{\max}^+ X^{n+1}\]
for every $k\in \Z$. We conclude that for every $P\in \Z[X, X^{-1}]$, we have
\[h_{n+1}^{-1}.t(P)< h_n^k.t(P)< h_{n+1}.t(P)\]
for every $k\in \Z$, and thus ${I(n, P)}$ is compactly contained in $I(n+1, P)$.
\end{proof}

After Claim \ref{l-plante-fixed-points}, we have that the collection $\mathcal{L}=\bigcup_n \mathcal{L}_n$ defined above is a prelamination. Next, we check $\varphi$-invariance: by commutativity, the image of the generator $\varphi(h_0)$ preserves each support $\suppphi(h_n)$ by commutativity, so each $\mathcal{L}_n$ is $\varphi(h_0)$ invariant; for the other generator, we have
\[g.\suppphi(h_n)=\suppphi(gh_ng^{-1})=\suppphi(h_{n+1}),\]
so $g.\mathcal{L}_n=\mathcal{L}_{n+1}$. This also gives that the map $\hor$ intertwines the $\varphi$-action on $\mathcal{L}$ with the cyclic action $j$ defined above, defining a positive horograding.

Finally, let us check that $\mathcal{L}$ is a discrete (hence closed) subset of $\R^{(2)}$. Indeed, suppose that $(l_n)\subset \mathcal{L}$ is a sequence converging to some interval $I\in \R^{(2)}$. Up to discard finitely many terms, this gives that the intervals $l_n$ are all related by inclusion, and upon extracting a subsequence, we can assume that the sequence $(l_n)$ is monotone, and thus also $(\hor(l_n))$ is. Now, if the sequence $\hor(l_n)$ is bounded, then Claim \ref{l-plante-fixed-points} implies that $(l_n)$ is eventually constant and equal to $I$, in particular $I\in \mathcal{L}$. If $\hor(l_n)$ increases to $+\infty$, then $I=\bigcup l_n$ is a connected component of $\bigcup_{m\in \Z}\suppphi(h_m)$; however, as the latter is a $\varphi$-invariant open set, minimality of $\varphi$ (Proposition \ref{prop Plante-like minimal}) gives $I=\R$, contradicting that $I$ is a bounded open interval. Finally, if $\hor(l_n)\to -\infty$, we have that $I$ is a connected component of the interior of $\bigcap_{m\in \Z}\suppphi(h_m)$; using minimality again, we conclude that $I=\varnothing$. Thus $\mathcal{L}$ is a discrete subset of $\R^{(2)}$.
\end{proof}

\subsubsection{Plante-like actions of more general wreath products} \label{s-wreath}
As explained in \cite[Example 8.1.8]{BMRT}, the construction of Plante-like actions of $\Z \wr \Z$ can be easily generalized to more general wreath products.   Recall that given groups $B$, $H$, and an action of $B$ on a set $\Omega$, the \emph{permutational wreath product} $H\wr_\Omega B$ is the semi-direct product $(\bigoplus_\Omega H)\rtimes B$, where $\bigoplus_\Omega H$ is the group of finitely supported configurations $f\colon \Omega \to H$, and $B$ acts on it via the shift action $b\cdot f(\omega)=f(b^{-1}\omega)$.  (When $\Omega=B$ and the action is the left-regular action, this gives the wreath product $H\wr B$.)
The group $H\wr_\Omega B$ then acts naturally on $\bigoplus_\Omega H$, by letting the group $\bigoplus_{\Omega} H$ act on itself by translations, and the group $B$ act on $\bigoplus_{\Omega} H$ by the shift action. 

Suppose that $\prec_\Omega$ is a $B$-invariant order on $\Omega$, and that $\prec_H$ is a left-invariant order on the group $H$. This allows to define an order of lexicographic type $\prec$ on $\bigoplus_\Omega H$, whose positive cone is the set  of configurations $f\in \bigoplus_\Omega H$ such that  $f(x_f)\succ_H \id$, where $x_f:=\max_{\prec_{\Omega}}\{x\in \Omega :  f(x)\neq \id\}.$ The order $\prec$ is invariant for the natural action of $H\wr_\Omega B$ on $\bigoplus_{\Omega} H$. If $\Omega$ and $H$ are countable, and $(\Omega,\prec_\Omega)$ is unbounded below and above, we obtain an action $\varphi\colon H \wr_\Omega G\to \homeo_0(\R)$ by taking the dynamical realization of this action, that we call a \emph{Plante-like} action of $H \wr_\Omega G$. This construction leads to the following result.
In the statement we denote by $\pi_B\colon H\wr_\Omega B\to B$ the natural quotient projection. 

\begin{prop}\label{p-plante-like-general}
 Let $\rho\colon B\to \homeo_0(\R)$ be an irreducible action of a countable group $B$, and $\Omega\subset \R$ be a countable $\rho$-invariant set. Let $H$ be any countable left-orderable group. Then there exists a faithful minimal laminar action $\varphi\colon H\wr_{\Omega} B\to \homeo_0(\R)$, horograded by $\rho\circ \pi_B$. 
\end{prop}

We refer to \cite[Example 8.1.8]{BMRT} for the proof and further details on the construction. Although the general result will not be used in this article, we will consider a special case in  \S \ref{ssc.ZwrZwrZ}.  For the moment, let us simply observe that Proposition \ref{p-plante-like-general} can be used iteratively to produce minimal laminar  faithful actions of solvable groups of arbitrary large solvability degree, and examples of actions that are horograded by more complicated actions than a cyclic one. Many more examples of laminar actions and horogradings can be found in \cite{BMRT}.

\subsection{Minimal laminations}

\begin{dfn} Let $\mathcal{L}\subset \R^{(2)}$ be a non-empty lamination preserved by an action $\varphi\colon G\to\homeo_0(\R)$. We say that $\mathcal{L}$ is a \emph{minimal} $\varphi$-invariant lamination if the subset $\mathcal{L}$ is minimal with respect to inclusion, among the family of non-empty $\varphi$-invariant laminations. 
\end{dfn}

In other terms, $\mathcal{L}$ is a minimal $\varphi$-invariant lamination if and only if the subset $\mathcal{L}\subset\R^{(2)}$ is a minimal closed  invariant subset for the diagonal action induced by $\varphi$ on $\R^{(2)}$.  When there is no risk of confusion, we will simply adopt the terminology ``minimal lamination'' instead of  ``minimal $\varphi$-invariant  lamination''.
As discussed in \S \ref{sc.preliminary-actions}, every finitely generated subgroup of $\homeo_0(\R)$ admits a minimal invariant subset. Our first general result is that, in analogy to minimal invariant subsets for actions on the real line, minimal laminations always exist for focal laminar actions of finitely generated groups.

\begin{prop}\label{prop.minimalamination} Let $G$ be a finitely generated group, and $\varphi\colon G\to\homeo_0(\R) $ a focal laminar action with respect to  a lamination $\mathcal{L}$. Then $\mathcal{L}$ contains a minimal $\varphi$-invariant  lamination. 
\end{prop}
\begin{proof} Consider  a finite symmetric generating system $S=\{g_1,\ldots,g_n\}$ for $G$, and let $\|\cdot\|_S$ be the associated word length. As $\Lcal$ exhausts the real line and $S$ is finite, we can take an interval $l\in\mathcal{L}$ so that $h.l\cap l\neq\emptyset$ for every $h\in S$. The non-crossing condition gives that for every $h\in S$, either $h.l$ or $h^{-1}.l$ contains $l$. This implies that the subset $L:=\bigcup_{h\in S} h.l$ is an interval which belongs to $\mathcal{L}$. We can then consider the subset $\mathcal K:=\{k\in\mathcal{L}:l\subseteq k\subseteq L\}$, which is a compact subset of $\Lcal$.
We want to show that any orbit in $\mathcal{L}$ intersects $\mathcal K$. 
Take $k_0\in\mathcal{L}$. If $k_0\in \mathcal K$, there is nothing to prove. Otherwise, by definition of focal laminar action we can find an element $h\in G$ such that  $L\subseteq h.k_0$, and such that $n:=\|h\|_S$ is minimal among lengths of elements with this property. 
Write $h=f_n\cdots f_1$, with $f_i\in S$ for every $i\in \{1,\ldots,n\}$. By the minimality assumption on $h$, the interval $k_*=f_n^{-1}h.k_0$ does not contain $L$. On the other hand, by definition of $L$, we have $f_n.l\subseteq L$, and since $L\subseteq h.k_0 = f_n.k_*$, we must have $l\subseteq k_*$. As $l\subseteq L$, this gives $k_*\cap L\neq \varnothing$, so that by the non-crossing condition we must have $l\subseteq k_*\subseteq L$, or equivalently $k_*\in \mathcal K$, as desired.

	Consider now the families
	\[\mathcal{C}=\{C\subset\mathcal{L}:C\text{ is }\varphi\text{-invariant and closed}\}\quad\text{and}\quad\mathcal{C}_{\mathcal K}=\{C\cap \mathcal K:C\in\mathcal{C}\}.\]
	Denote by $i\colon \mathcal{C}\to\mathcal{C}_{\mathcal K}$ the map given by $C\mapsto C\cap \mathcal K$. Since every $\varphi$-orbit meets $\mathcal K$, we have that $i$ is a partial-order isomorphism with respect to inclusion. On the other hand, since $\mathcal K$ is compact, every decreasing chain in $\mathcal{C}_{\mathcal K}$ has an infimum, so that {after Zorn's lemma} we can find a $\varphi$-invariant lamination which is minimal with respect to inclusion in $\mathcal{C}$. 
\end{proof}

\begin{rem}\label{r-discrete-minimal}
Minimal invariant laminations naturally split into two families: the \emph{discrete} and the \emph{non-discrete} ones (with respect to the natural topology induced from $\R^{(2)}$). 
When $\mathcal{L}$ is discrete, each $\varphi$-orbit in $\mathcal{L}$ is open, and by minimality we have that $\mathcal{L}$ consists of a single $\varphi$-orbit. Moreover, for each leaf $l_0\in \mathcal{L}$, the set $\{l\in \mathcal{L} : l\supseteq l_0\}$ is totally ordered and discrete, so  isomorphic to $\N$ as an ordered set. In particular, every $l_0$ has a successor $l_1$ in the poset $(\mathcal{L}, \subseteq)$.
When $\mathcal{L}$ is non-discrete, minimality implies that it has no isolated points.\end{rem}

For later use, we record a simple lemma on horogradings defined on minimal laminations.

\begin{lem} \label{l-horograding-almost-injective}
Let $\varphi\colon G\to \homeo_0(\R)$ be a minimal laminar action, and $(\mathcal{L}, \hor)$ a horograding of $\varphi$ by  $j\colon G\to \homeo_0(\R)$, with $\mathcal{L}$ a minimal $\varphi$-invariant lamination. Suppose that $l_0,l_1\in \mathcal L$ are such that $l_0\subsetneq l_1$ and $\hor(l_0)=\hor(l_1)$. Then $l_1$ is an immediate successor of $l_0$. As a consequence, $\mathcal{L}$ is non-discrete, and $l_0$ and $l_1$ are in distinct $\varphi$-orbits.
\end{lem}
\begin{proof} We will assume that $(\mathcal L,\hor)$ is a positive horograding, the other case being analogous.
Let $l_0, l_1\in \mathcal L$ be as in the statement and suppose, by contradiction, that the set $U=\{k\in \mathcal{L}: l_0\subsetneq k\subsetneq l_1\}$ is non-empty. Note that $U$ is an open, totally ordered subset of $\mathcal{L}$, on which the horograding $\hor$ is constant, since $\hor(l_1)\leq \hor(k)\leq \hor(l_2)=\hor(l_1)$ for any $k\in U$. By minimality, the orbit of every $k\in \mathcal{L}$ intersects $U$. From this, it follows that for every $k$,  there exists $k'\supsetneq k$ such that $\hor(k)=\hor(k')$. For a fixed $k$, the set of such $k'$, must be bounded above in $(\mathcal{L}, \subseteq)$ (see Remark \ref{r-horograding-proper}), thus
\[k_1:=\bigcup_{k'\in \mathcal L,\, \hor(k')=\hor(k)}k'\]
is a well-defined element of $\mathcal{L}$. We have $\hor(k_1)\neq \hor(k)$, as otherwise the previous reasoning yields some $k_1'\subsetneq k_1$ such that $\hor(k_1')=\hor(k)$, contradicting the choice of $k_1$. In particular, $k_1$ is accumulated by leaves $k'\subset k_1$ with $\hor(k')<\hor(k_1)$. But then, the same is true for all points in the $\varphi$-orbit of $k_1$, and since the latter intersects $U$, this is a contradiction.
Hence $U$ is empty, i.e.\ $l_0$ is an immediate successor of $l_1$.

Now, if $\mathcal L$ was discrete, then it would consist of a single orbit (Remark \ref{r-discrete-minimal}), and equivariance of $\hor$ would easily imply that $\hor$ is constant, contradicting Remark \ref{r-horograding-proper}.
Hence $\mathcal{L}$ is non-discrete, and thus has no isolated point (Remark \ref{r-discrete-minimal}). Since the leaf $l_0$ has a successor, it is not accumulated by leaves that contain it, and thus must be accumulated by leaves contained in it, and similarly $l_1$ must be accumulated by leaves that contain it. It follows that  no element of $G$ can map $l_0$ to $l_1$. 
\end{proof}

\subsection{Minimal actions on directed trees} In this subsection we recall some terminology from \cite[Chapter 11]{BMRT} about group actions on directed trees, and introduce an adequate notion of minimality for such actions.

\begin{dfn}\label{def d-trees}
	A  \emph{directed tree} is a poset $(\Tbb, \treeorder)$ with the following properties.
	\begin{enumerate}[label=(T\arabic*)]
		
		\item \label{i-tree} For every $v\in \Tbb$, the subset $\{u\in \Tbb : v \treeordereq u\}$  is totally ordered and order-isomorphic to the half-line $[0, +\infty)$.

		\item \label{i-directed} Every pair of points  $u, v\in \Tbb$ has a smallest common upper bound, denoted as $u\treeup v$.
		
		\item \label{i-chains-R} There exists a countable subset $\Sigma \subset \Tbb$ such that for every  distinct $u, v\in \Tbb$ with $u \treeordereq v$ there exists $z\in \Sigma$ such that $u\treeordereq z\treeordereq v$.
	\end{enumerate}

For later use in \S \ref{s-modding-out}, we also record here the following variant of this notion.
\begin{dfn} \label{d-pre-directed-tree} We say that a poset $(\Tbb,\treeorder)$ is a \emph{pre-directed tree} if it satisfies conditions  \ref{i-tree} and \ref{i-chains-R} in the definition of directed tree and 
	\begin{enumerate}[label=(T2')]
		\item \label{c-predirected} for every pair of points  $u, v\in \Tbb$ there is $w\in \Tbb$ such that $u \treeorder w$ and $v\treeorder w$. 
	\end{enumerate}
	  \end{dfn}

\begin{rem} \label{r-Hausdorff} To understand the difference between \ref{i-directed} and \ref{c-predirected}, think of  the two posets obtained by gluing two copies of $(\R, <)$ along the closed ray $[0, +\infty)$, or along the open ray $(0, +\infty)$, respectively. The former is a directed tree, while the latter is only a pre-directed tree.  Although we will not consider any topology on a directed tree, one may think of a pre-directed tree as a directed tree  which is possibly not Hausdorff.  \end{rem}
\begin{dfn} \label{d-tree-horograding}

We say that a point $u\in \Tbb$ is \emph{below} $v\in \Tbb$ (or that $v$ is \emph{above} $u$) if $u\neq v$ and $u\treeordereq v$, and write $u\treeorder v$.	When $u$ is below $v$, we write $[u,v]=[v,u]=\{w\in \Tbb:u\treeordereq w \treeordereq v\}$ for the arc between $u$ and $v$, and for general $u$ and $v$ we set $[u, v]=[u, u\treeup v]\cup [v, u\treeup v]$.  We also write $]u,v[=[u,v]\setminus \{u,v\}$, and similarly we define $]u,v]$ and $[u,v[$. Arcs allow to define path-components and thus to introduce the notion of \emph{branching points}, that is, points $v\in \Tbb$ such that $\Tbb\setminus \{v\}$ has at least three path-components. We write $\Br(\Tbb)$ for the collection of branching points. We say that a directed tree $(\Tbb,\treeorder)$ is \emph{simplicial} if $\Br(\Tbb)\cap [v,w]$ is finite for every  $v,w\in \Tbb$. A point $v\in \Tbb$ is called a \emph{leaf} if it $\Tbb\setminus\{v\}$ has only one path-component, or equivalently, if $v$ is minimal with respect to $\treeorder$. We will be only interested in trees without leaves (as justified by Remark \ref{r-no-leaves} below). 
\end{dfn}

We will also use the notation $[v, \infty_\Tbb[=\{w\in \Tbb: v\treeordereq w \}$, and $]v, \infty_\Tbb[:=[v, \infty_\Tbb[\setminus\{v\}$. Here the notation $\infty_\Tbb$ stands for an imaginary extra point, which is greater than any element in $\Tbb$, and will be called the \emph{focus}.

We next discuss the dynamics of group actions on directed trees.

\begin{dfn}
	Let $G\subseteq\Aut(\Tbb,\treeorder)$ be a group of automorphisms of a pre-directed tree.
	
	\begin{itemize}
		\item $G$ is \emph{focal} if for any $u,v\in \Tbb$ there exists $g\in G$ such that $v\treeorder g.u$. 
		\item The action of $G$ is \emph{d-minimal} if for every $v_1, v_2,w_-,w\in\Tbb$ with $v_1\treeorder v_2$ and $w_-\treeorder w$, there exist $g\in G$ and $w'\in [w_-,w[$ such that $g.[w',w]\subset]v_1,v_2[$. 
		\item The action of $G$ is \emph{simplicial} if $(\Tbb,\treeorder)$ is simplicial and for every $v,w\in\Br(\Tbb)$ the stabilizer $\stab^\Phi (\{v,w\})$ acts trivially on $[v,w]$. 
	\end{itemize}
When $\varphi\colon G\to\Aut(\Tbb,\treeorder)$ is an action on a directed tree, we say that $\Phi$ is \emph{focal}, \emph{d-minimal}, or \emph{simplicial}, according to what $\Phi(G)$ is.
\end{dfn}

\begin{rem}\label{r.dminimal_is_focal} Any d-minimal action (where the letter ``d'' stands for \emph{directionally}) is focal. Indeed, the definition says that it is possible to send every point $w$ inside any interval $[v_1, v_2[$, and moreover this can be done so that any given direction below $w$ (defined by $w_-$) is mapped to the direction defined by the point $v_1$. \end{rem}

\begin{rem} \label{r-no-leaves}
If $(\Tbb, \treeorder)$ admits a focal group of automorphisms, then it has no leaves. 
\end{rem}

Let us now recall from \cite[\S 11.1.6]{BMRT} the notion of horograding for actions on directed trees. 

\begin{dfn}
A \emph{(positive) horograding} of a directed tree $(\Tbb,\treeorder)$ is an increasing map $\pi\colon (\Tbb, \treeorder)\to (\R,<)$ such that for every $u,v\in \Tbb$ verifying $u\treeorder v$, the restriction of $\pi$ to the arc $[u,v]\subset \Tbb$ is an order-preserving bijection onto the interval $[\pi(u),\pi(v)]\subset \R$. Given an action $\Phi\colon G\to \Aut(\Tbb,\treeorder)$, we say that $\Phi$ is \emph{horograded} by an action $j\colon G\to \homeo_0(\R)$ if there exists a $G$-equivariant surjective horograding $\pi\colon\Tbb\to \R$. 
\end{dfn}

A relevant special case of horograding arises in the situation when the action of $G$ on $(\Tbb, \treeorder)$ can be chosen to be isometric with respect to an $\R$-tree metric. Recall that an  $\R$\emph{-tree} is a metric space $(X,d)$ where every pair of points $v,w\in X$ can be joined by a unique arc, and this arc can be chosen to be a geodesic. We say that a directed tree $(\Tbb,\treeorder)$ has a \emph{compatible} metric if $\Tbb$ is endowed with an $\R$-tree metric, for which the subsets of the form $[v,w]$ associated with $\treeorder$ coincide with the geodesic segments. In \cite[\S 11.3.1]{BMRT} it is shown that every directed tree in the sense of Definition \ref{def d-trees} can be endowed with such a compatible $\R$-tree metric. However, in the presence of an action $\varphi\colon G\to\Aut(\Tbb,\treeorder)$ it is not always possible to find $\Phi $-invariant compatible metrics. The following proposition describes when it is the case in terms of the existence of abelian horogradings. Its proof is contained in \cite[Proposition 11.3.1]{BMRT} and \cite[Proposition 11.3.3]{BMRT}. 

\begin{prop}\label{prop.Rtreemetric} Let $\Phi\colon G\to\Aut(\Tbb,\treeorder)$ be a focal action on a directed tree. Then, $\Phi $ preserves a compatible metric on $(\Tbb,\treeorder)$ if and only if $\Phi $ can be horograded by an action by translations $j\colon G\to (\R, +)$. Moreover, if $\Phi $ is simplicial, then $\Phi $ is horograded by a cyclic action. 
\end{prop}

\begin{proof}[Sketch of proof in the simplicial case] 
The simplicial case of the previous proposition will naturally arise in the proof Theorem \ref{t-main}, so let us briefly recall why this implication is true. If the directed tree $\Aut(\Tbb, \treeorder)$ is simplicial, it can be endowed with a natural simplicial distance, obtained by identifying isometrically $[v, w]$ with the interval $[0, 1]$ whenever $v\treeorder w$ are adjacent branching points (that is, branching points such that $]v, w[\cap \Br(\Tbb)=\varnothing$). Fix $v\in\Tbb$, and define $\pi\colon\Tbb\to\R$ as 
\[\pi(w)=d(v,v\treeup w)-d(w,v\treeup w).\]
Notice that $\pi$ is the unique horograding satisfying $\pi(v)=0$, and such that the restriction of $\pi$ to any ray of the form $[w,\infty_\Tbb[$ is an order-preserving isometry. 
Given $g\in G$, choose $u\in \Br(\Tbb)$ such that $g.u$ and $u$ are $\treeorder$-comparable (for instance, this is true for $u=g^{-1}.v \treeup v$ for an arbitrary $v\in \Tbb$), and set $t_g= d(g.u, u)$. Note that  $g$ moves every point in the ray $[u, \infty_\Tbb[$ by a distance $t_g$. It follows that $g$ eventually  acts as a translation by $t_g$ along any ray $[w, \infty_\Tbb[$. Thus  $t_g$ does not depend on the chosen $u$, and the map $g\mapsto t_g$ is a homomorphism  $G\to \Z$. It is also easy to check that 
$\pi(g.w)=\pi(w)+t_g$
for every $w\in \Tbb$. Thus $\pi$ is a horograding of $\Phi $ by the cyclic action defined by $j(g) \colon x\mapsto x+t_g$.\qedhere
\end{proof}

 \subsection{From minimal laminations to directed trees}
 \label{subsection trees}
 
 The next result is a more precise version of \cite[Proposition 11.2.3]{BMRT} for minimal laminations.

\begin{prop}\label{prop.minimalmodeltree} Let $G$ be a finitely generated group and let $\varphi\colon G\to\homeo_0(\R)$ be a minimal laminar action. Let $\mathcal{L}$ be a minimal $\varphi$-invariant lamination. Then, there exist an action $\Phi\colon G\to\Aut(\Tbb,\prec)$  on a directed tree and a non-decreasing $G$-equivariant map  $\iota \colon (\mathcal{L}, \subset)\to (\Tbb, \treeorder)$ satisfying the following:
\begin{itemize}
\item if the lamination $\mathcal{L}$ is discrete, then $\Phi $ is simplicial and has only one orbit of branching points;

\item if the lamination $\mathcal{L}$ is non-discrete, then $\Phi $ is  d-minimal.

\end{itemize}

\end{prop}

In the proof, we will use a well-known  result of Cantor  (see Jech \cite[Theorem 4.3]{Jech}): up to isomorphism, $(\R, <)$ is the unique totally ordered set  that has no maximum or minimum, is Dedekind complete (i.e., every subset with a lower bound has a greatest lower bound), and admits a countable order-dense subset  (i.e.\ a countable subset $Q$ such that for every $a< b$, there exists $q\in Q$ such that $a< q< b$).
\begin{proof}[Proof of Proposition \ref{prop.minimalmodeltree}] We first treat the case where $\mathcal{L}$ is discrete. In this case,   the action of $G$ on $\mathcal{L}$ is transitive, and each leaf $l_0$ has a well-defined successor $l_1$ (see Remark \ref{r-discrete-minimal}).  We  consider the simplicial tree $\Tbb$  with vertex set $\mathcal{L}$, obtained by gluing an edge (isomorphic to $[0,1]$) between every leaf and each successor, and let $\iota\colon \mathcal{L}\to \Tbb$ be the inclusion. We endow $\Tbb$ with the natural order $\treeorder$ that extends the order $\subseteq$ on $\mathcal{L}$. It is easy to check that $(\Tbb, \treeorder)$ is a simplicial directed tree in the sense of Definition \ref{def d-trees}.  Moreover, there is a unique simplicial action $\Phi\colon G\to \Aut(\Tbb, \treeorder)$ that extends the action on $\mathcal{L}$, which is focal since the action on $\mathcal{L}$ is transitive. Finally note that each branching point of $\Tbb$ is in $\iota(\mathcal{L})$, and since the action on $\mathcal{L}$ is transitive, $\Phi $ is transitive on branching points.

Assume now that $\mathcal{L}$ is non-discrete, and hence has no isolated point (Remark \ref{r-discrete-minimal}). Say that a leaf $l\in\mathcal{L}$ is \emph{accumulated from above} if the subset $\{k\in\mathcal{L}:k\supsetneq l\}$ accumulates on $l$. Analogously we define when a leaf is \emph{accumulated from below}. 
Given $l\in\mathcal{L}$ write \[\textstyle l^\ast=\mathsf{Int}\left(\bigcap_{k\in\mathcal{L},\,k\supsetneq l}k\right),\] 
and note that $l^*\in \Lcal$ as $\Lcal$ is closed in $\R^{(2)}$.
\begin{claim}\label{claim.defstar} The subset $\mathcal{L}^\ast:=\{l^\ast:l\in\mathcal{L}\}$ coincides with the subset of leaves in $\Lcal$ accumulated from above, and the equality $(l^\ast)^\ast=l^\ast$ holds for every $l\in\mathcal{L}$. 
\end{claim}
\begin{proof}[Proof of claim]
	First notice that, by definition, if $l\in \Lcal$ is accumulated from above, then $l=l^\ast$ and therefore $l\in\mathcal{L}^\ast$. Take now $k=l^\ast$ in $\mathcal{L}^\ast$. If $l=k$ then, by definition again, $k$ is accumulated from above. On the other hand, when $l\neq k$, the leaf $k$ is not accumulated from below. Thus, since $\mathcal{L}$ is perfect, $k$ must be accumulated from above. Also notice that $k^\ast=k$ holds in both cases.
\end{proof} 

We want to show that $(\Tbb, \treeorder):=(\Lcal^*,\subset)$ is the desired directed tree, and the action $\Phi $ induced by $\varphi$ is the desired d-minimal action. The map $\iota$ in the statement will then be given by $\iota(l)=l^\ast$. Let us begin by verifying that $(\Lcal^*,\subset)$ is a directed tree. 

We start with condition \ref{i-chains-R}.  As a topological space, $\mathcal{L}$ is separable, thus it has a countable dense subset $\mathcal{Q}$. Set $\mathcal{Q}^\ast=\{q^\ast : q\in \mathcal{Q}\}$. Consider any two leaves $l_1\subsetneq l_2$ in $\mathcal{L}^\ast$. Using that $l_1$ is accumulated from above, we can find  $l_3\in \mathcal{L}^\ast$ such that $l_1\subsetneq l_3\subsetneq l_2$. Now, the set $\{k\in \mathcal{L}: l_1 \subsetneq k\subsetneq l_3\}$ is open in $\mathcal{L}$ and non-empty (using again that $l_1$ is accumulated from above), thus contains an element $q\in \mathcal{Q}$. It follows that $l_1 \subsetneq q^\ast \subseteq l_3\subsetneq l_2$; as $l_1\subsetneq l_2$ were arbitrary, this gives \ref{i-chains-R}.

We now check condition \ref{i-tree}: for fixed $l\in \Lcal^*$, we need to show that the subset $\mathcal R=\{k\in \Lcal^*:l\subsetneq k\}$ is order-isomorphic to $(\R, <)$. We do so by applying Cantor's characterization. The subset $\mathcal R$ has no maximum or minimum (since $l$ is accumulated from above), and $\mathcal R\cap \mathcal{Q}^\ast$ is a countable order-dense subset by the paragraph above. To check that $(\mathcal R, \subset)$ is Dedekind complete, let $S$ be a subset admitting a lower bound $k_0\in \mathcal R$. Define $k_1:=\Int\left (\bigcap_{k\in S}k\right )$. Then $k_1$ is a non-empty interval (as it contains $k_0$), and  $k_1\in\mathcal{L}$ since $\mathcal{L}$ is closed. If $k_1\in S$, then it is a minimum element of $S$. Else $k_1$ must be non-trivially accumulated from above, so $k_1\in \mathcal{L}^\ast$, and actually $k_1\in \mathcal R$ as $k_1\supseteq k_0\supsetneq l$. In either case, we conclude that $k_1$ is a greatest lower bound to $S$ in $\mathcal R$. Condition \ref{i-tree} follows.

Finally, we check condition \ref{i-directed}. For this, fix $l_1,l_2\in \Lcal^*$ and define 
\[\textstyle k_0:=\Int\left (\bigcap_{k\in\mathcal{L},\,(l_1\cup l_2)\subseteq k }k \right).\] We have $k_0\in \mathcal{L}$, and we need to show that $k_0\in \Lcal^*$, or equivalently, by the claim, that $k_0$ is accumulated from above. If not, since $\mathcal{L}$ is perfect, $k_0$ must be accumulated from below by leaves in $\Lcal^*$. Taking a leaf $k_0'\in \Lcal^*$ strictly contained in $k_0$ and sufficiently close to $k_0$, we must have $k_0'\cap l_i\neq \varnothing$ for $i\in \{1,2\}$, and by the non-crossing condition, this implies $k_0'\supset (l_1\cup l_2)$, contradicting the choice of $k_0$.

It remains to prove that the action $\Phi\colon G\to \Aut(\Lcal^*,\subset)$ induced by $\varphi$ is d-minimal.  For this, take leaves $l_1\subsetneq l_2$ and $k_1\subsetneq k_2$ in $\mathcal{L}^\ast$, and we want to find an element $g\in G$ and a leaf $k_3\in [k_1,k_2]$ such that $g.[k_3,k_2]\subset ]l_1,l_2[$. Taking a smaller $l_2$ if necessary, we can assume that $l_2$ is also accumulated from below. Consider the leaf $k_3'$ in $\Lcal$ given by
\[{k}_3':=\bigcup_{k\in\mathcal{L},\,k_1\subset k\subsetneq k_2}k,\]
and note that $ k_3'$ is accumulated from below. Again, by the claim and minimality of $\Lcal$, there exists an element $g\in G$ such that
$l_1\subsetneq g. k_3' \subsetneq l_2$. As $l_2$ is accumulated from below, and $k_2$ is either $k_3'$ or a successor of $k_3'$ in $(\mathcal{L}, \subset)$, we must have $l_1\subsetneq g.k_2 \subsetneq l_2$. As $k_3'$ is accumulated from below, we can also find a leaf $k_3\in \Lcal^*$ such that $l_1\subsetneq g.k_3\subsetneq g. k_3'$. This gives the desired conclusion.
\end{proof}

\begin{rem}\label{r-tree-horograding}
From Proposition \ref{prop.minimalmodeltree}, it follows  that in order to find a horograding of the laminar action $\varphi$ by some $j\colon G\to \homeo_0(\R)$, it is enough to find a horograding of the associated action $\Phi \colon G\to \Aut(\Tbb, \treeorder)$ by $j$. Indeed if $\pi\colon \Tbb\to \R$ is such a horograding, then $\hor:=\pi\circ \iota\colon \mathcal{L}\to \R$ defines a horograding of $\varphi$ by $j$.
\end{rem}

\section{Horocyclic subgroups in focal actions} 
\label{sec horocyclic}

In this section we explain a technical mechanism that allows, given a d-minimal action $\Phi  \colon G\to \Aut(\Tbb, \treeorder)$  on a directed tree and a suitable normal subgroup $N$ of $G$, to ``mod out $N$'' and  construct an action of $G/N$ on a smaller  directed tree. This mechanism will be the main tool to study laminar actions of solvable groups, by successively simplifying them.
\subsection{Horocyclic subgroups}

\begin{dfn} 
	An automorphism $g\in \Aut(\Tbb,\treeorder)$ of a directed tree is \emph{elliptic} if for every $v\in\Tbb$ verifying $g.v\treeordereq v$, one has $g.v=v$.
	A subgroup $G\leq \Aut(\Tbb,\treeorder)$ is {horocyclic}, if it acts without fixed points and contains only elliptic elements. Similarly, when $\Phi\colon G\to \Aut(\Tbb,\treeorder)$ is an action such that $\Phi(G)$ is horocyclic, we say that $\Phi$ is {horocyclic}.
\end{dfn}

\begin{rem}\label{r.isometries}
	When the directed tree $(\Tbb,\treeorder)$ is simplicial (or more generally, when an action preserves an $\R$-tree metric), the definition above corresponds to the classical notion of elliptic isometry. Indeed, for isometric actions on simplicial trees, every automorphism is either elliptic (which is equivalent to having a fixed point in $\Tbb$) or \emph{hyperbolic}. Recall that hyperbolic automorphisms admit a unique translation axis (an embedded line on which the automorphism acts as a translation), and in the case of directed trees, the axis contains the focus in its boundary.
\end{rem}

\begin{prop}\label{p-focal-normal}
	Let $G\leq \Aut(\Tbb, \treeorder)$ be a subgroup of automorphisms of a directed tree, whose action is either simplicial and transtive on branching points, or d-minimal. Then any non-trivial normal subgroup of $G$ is either focal or horocyclic. \end{prop}

\begin{proof}
	Let $N\lhd G$ be a normal subgroup. First note that  $N$ acts without fixed points. Indeed note that the set of fixed points of $N$ is $G$-invariant, and assume that it is non-empty. In the d-minimal case, $G$-invariance implies that it must intersect every arc $]v, w[$, and it follows that $N$ must be trivial. In the simplicial case, the set of fixed points of $N$ is actually a simplicial subtree, and since the $G$-action is transitive on branching points, we also deduce that $N$ is trivial. In either case, since $N$ is non-trivial, it has no fixed points.
	
	Consider first the case where the action of $G$ is  simplicial.  We distinguish two cases according to whether the image of $N$ contains a hyperbolic automorphism of $\Tbb$ or not  (see Remark \ref{r.isometries}). If $N$ does not contain any hyperbolic element, then it is horocyclic. If $N$ does contain a hyperbolic element, let $\mathbb S$ be the union of the axes of its hyperbolic elements. Note that $\mathbb S$ is a directed subtree, and the action of $N$ on $\mathbb S$ is focal. Indeed, if $v\in \mathbb{S}$, then there exists a hyperbolic element $g\in G$  whose axis contains $v$, and thus it  necessarily contains the ray $[v,\infty_\Tbb[$ (since the point $\infty_\Tbb$ can be interpreted as an end of the simplicial tree $\Tbb$ fixed by the group $G$). As $N$ is normal, $\mathbb S$ is also $G$-invariant. Finally, take a point $v\in \mathbb S$ and note that as $G$ is focal, for any $u\in \Tbb$ there exists $g\in G$ such that $g.u\in [v,\infty_\Tbb[\subset \mathbb S$. This shows that $\Tbb=\mathbb S$, and we have already observed that the action of $N$ on $\Tbb=\mathbb S$ is focal, as desired.

	Consider next the case where the action of $G$ is d-minimal and assume by contradiction that $N$ is neither horocyclic nor focal.
	
	\begin{claim}
		For any $v\in \Tbb$ and $v'\treeorder v$, there exist $k\in N$ and $w\in [v', v[$ such that $v\treeorder k.w$.
	\end{claim}
	
	\begin{proof}
		Since $N$ is not horocyclic, we can choose a point $v_0\in \Tbb$ and $h\in N$ such that $v_0\treeorder h.v_0$. By d-minimality, for given $v'\treeorder v$, we can find $w\in [v', v[$ and $g\in G$ such that $g.[w,v]\subset\, ]v_0,h.v_0[$. By the choice of $h$, we must then have $g.v\treeorder h.v_0\treeorder hg.w$. Applying $g^{-1}$, we get $v\treeorder g^{-1}hg.w$. As $N$ is normal and $h\in N$, the element $k:=g^{-1}hg\in N$ is the element  we were looking for.
	\end{proof}
	
	Now since $N$ is not focal,  there exists $u\in \Tbb$ such that $N.u\cap [u, \infty_\Tbb[$ is bounded above; set $v=\sup_{\treeorder} (N.u\cap [u, \infty_\Tbb[)$. After the claim, we can take $k\in N$ and $w\in [u, v[$ such that $v\treeorder k.w$. Note that the relation  above  holds for any $w'\in [w,v]$. In particular, for any $w'\in N.u\cap [w,\infty_\Tbb[$, we will have $v\treeorder k.w'$, but as $k.w'\in N.u\cap [u,\infty_\Tbb[$ and $v=\sup_{\treeorder} (N.u\cap [u, \infty_\Tbb[)$, we have reached a contradiction.
\end{proof}

From the previous proposition, we obtain the following.

\begin{cor}\label{c-normal-focal-centralizer}
	Let $G\leq \Aut(\Tbb,\treeorder)$ be a focal subgroup of automorphisms of a directed tree, and let $N$ be a non-trivial normal subgroup of $G$ whose centralizer in $G$ is non-trivial. Assume further that the action of $G$ is either simplicial or d-minimal, and that $\Tbb\neq\R$. Then, $N$ is horocyclic. 
\end{cor}

The proof of the corollary is also based on the next general result for focal subgroups.

\begin{lem}\label{l-centralizer-focal}
	Let $G\leq \Aut(\Tbb, \treeorder)$ be a focal subgroup of automorphisms of a directed tree $\Tbb\neq\R$. Assume that the action of $G$ is either d-minimal or simplicial. Let $H<G$ be a focal subgroup of $G$. Then the centralizer of $H$ in $G$ is trivial.
\end{lem}
\begin{proof} By contradiction, we assume that there exists a non-trivial automorphism $k\in G$ centralizing $H$. We first claim that there exists a point $v\in\Tbb$ with $v$ and $k.v$ that are not $\treeorder$-related. To see this, assume first that the action of $G$ is d-minimal. This assumption, together with the fact that $\Tbb\neq \R$, implies that every non-empty open path $]v, w[\subset \Tbb$ contains branching points.  Choose $w\in \Tbb$ such that $k.w\neq w$. If $w$ and $h.w$ are not $\treeorder$-related, we choose $v=w$. Else assume, say, that $w\treeorder k.w$, and let $v_0\in ]w, k.w[$ be a branching point. Let $v\treeorder v_0$ lie in a different path-component of $\Tbb \setminus\{v_0\}$ as $w$. Then $v$ and $k.v$ are not $\treeorder$-related. If the action is simplicial, choose $w\in \operatorname{Br}(\Tbb)$ such that $k.w \neq w$, and let $v_1$ and $v_2$ be two distinct branching points below $w$. Then at least one $v\in \{v_1,v_2\}$ is such that $k.v$ and $v$ are not $\treeorder$-related.

Let now $v$ be as in the claim. Write $w:=v\treeup k.v$. Since $H$ is focal, we can take $g\in H$ so that $g.v\treeorderop w$. On the one hand, since $v$ and $k.v$ are not $\treeorder$-related, neither are $g.v$ and $gk.v$. Thus, since $k.v\treeorder w\treeorder g.v$, we conclude that $k.v$ and $gk.v$ are non $\treeorder$-related. On the other hand, since $v\treeorder w\treeorder g.v$ we have that $k.v\treeorder kg.v=gk.v$ which contradicts the previous conclusion. \qedhere
\end{proof}

\begin{proof}[Proof of Corollary \ref{c-normal-focal-centralizer}]
	Combine Lemma \ref{l-centralizer-focal} with Proposition \ref{p-focal-normal}.
\end{proof}

\subsection{Modding out horocyclic normal subgroups}
\label{s-modding-out}
The aim of this section is to explain that if $\Phi\colon G\to \Aut(\Tbb, \treeorder)$ is an action and $N$ is a normal horocyclic subgroup, then $N$ can be mod out to obtain a new action of $G/N$ on a directed tree (Proposition \ref{l-quotienting-horocyclic} below). The quite intuitive idea is to consider the action on the quotient $\Tbb/N$. The issue however is that the space $\Tbb/N$ is not a directed tree in general, but only a pre-directed tree as defined in Definition \ref{d-pre-directed-tree}   (see Example \ref{e-predirected} below). 
Fortunately, this is only a minor technicality that can be readily solved by passing to a slightly smaller quotient.  This is the goal of the discussion below.

	Let $(\Tbb_1,\treeorder)$ and $(\Tbb_2,\treeorder)$ be pre-directed trees. We say that a map $\pi\colon \Tbb_1\to\Tbb_2$ is a \emph{grading} if it is surjective and for every $v, w\in \Tbb_1$ such that $v\treeorder w$, we have that $\pi(v)\treeorder \pi(w)$ and the restriction of $\pi$ to $[v, w]$ is an order-preserving isomorphism onto $[\pi(v), \pi(w)]$. If in addition, for $i\in \{1,2\}$, one has that $\Phi_i\colon G\to \Aut(\Tbb_i, \treeorder)$ are order-preserving actions admitting a $G$-equivariant grading from  $\Tbb_1$ onto $\Tbb_2$, we say that $\Phi_1$ can be \emph{graded} by $\Phi_2$.
\end{dfn}

\begin{lem}\label{lem.predirectedtree} Let $\varphi\colon G\to \Aut(\Tbb,\treeorder)$ be an action on a pre-directed tree. Then, there exists an action $\widehat\varphi\colon G\to\Aut(\widehat\Tbb,\treeorder)$ on a directed tree, grading  $\Phi $. Moreover, if $\Phi $ is d-minimal so is $\widehat\Phi $. 
\end{lem}
\begin{proof} Given $v\in\Tbb$, write $\widehat v:=\{w\in\Tbb:v\treeorder w\}$ and $\widehat\Tbb:=\{\widehat v:v\in\Tbb\}$. We set $\widehat v\treeorder\widehat w$ if and only if $\widehat v\supsetneq \widehat w$. Equivalently, $\widehat w\treeorder \widehat v$ if and only if $w\in \widehat v$. It is direct to check that if $\Sigma\subset\Tbb$ is a countable subset given by condition \ref{i-chains-R} for $(\Tbb,\treeorder)$, then the subset $\widehat\Sigma:=\left \{\widehat v\in \widehat{\Tbb}:v\in\Sigma\right \}$ also satisfies condition \ref{i-chains-R}. In order to check condition \ref{i-tree}, for any $v\in \Tbb$ write  
	\begin{equation}
		\label{eq.Lv}L_{v}=\left \{\widehat w\in\widehat\Tbb:\widehat v\treeorder\widehat w\right \}=\left \{\widehat w\in \widehat\Tbb:w\in \widehat v\right \}.
	\end{equation}
From this, it follows that the poset $(\widehat\Tbb,\treeorder)$ satisfies condition \ref{i-tree} in the definition of directed tree and that the map $\pi\colon \Tbb\to\widehat\Tbb$ defined by $\pi(v)=\widehat v$ is a grading. Notice that the action $\Phi $ projects to an action $\widehat\Phi\colon G\to\Aut(\widehat\Tbb,\treeorder)$ that grades $\Phi $ through the projection $\pi$. It also follows from \eqref{eq.Lv} that whenever $\Phi $ is d-minimal, so is $\widehat\Phi $. It remains to show that $(\widehat\Tbb,\treeorder)$ satisfies condition \ref{i-directed} in the definition of directed tree. For this, note that by condition \ref{c-predirected}, given $v_1,v_2\in\Tbb$, there exists $w\in\Tbb$ such that either $\widehat{v_1}\cap \widehat{v_2}=\widehat w$ or $\widehat{v_1}\cap \widehat{v_2}=\widehat w\cup \{w\}$. In both cases, the set of common upper bounds of $\widehat{v_1}$ and $\widehat{v_2}$ equals $L_{{v_1}}\cap L_{{v_2}}=L_{w}\cup\{\widehat w\}$ and therefore $\widehat w$ is the least common upper bound of $\widehat{v_1}$ and $\widehat{v_2}$, as wanted.
\end{proof}

\begin{prop} \label{l-quotienting-horocyclic}
	Let $\Phi\colon G\to \Aut(\Tbb, \treeorder)$ be an action on a directed tree. Assume that $N\unlhd G$ is a normal subgroup whose action is horocyclic. Then, $\Phi$ can be graded by an action $\widehat{\Phi}\colon G\to \Aut(\widehat{\Tbb}, \treeorder)$ on a directed tree where  $N$ acts trivially. Moreover, if $\Phi $ is d-minimal  so is $\widehat\Phi $.
\end{prop}
\begin{proof} We will first show that $\Phi $ is graded by an action on a pre-directed tree where $N$ acts trivially. For this, consider $\widehat{\Tbb}_0=\Tbb/\Phi(N)$, namely the quotient space of $\Tbb$ under the action of $N$, and denote by $\pi\colon \Tbb\to \widehat{\Tbb}_0$ the quotient projection. We define a partial order on $\widehat\Tbb_0$  by declaring $x_1\treeordereq x_2$ if there exist $v_i\in \pi^{-1}(x_i)$ for $i\in\{1,2\}$, with $v_1\treeordereq v_2$. In order to check that the relation is transitive, take $x_1,x_2,x_3\in\widehat\Tbb_0$ satisfying $x_1\treeordereq x_2\treeordereq x_3$. By definition of the partial order, there are points $v_i\in \pi^{-1}(x_i)$ for $i\in\{1,2\}$ and $w_i\in\pi^{-1}(x_i)$ for $i\in\{2, 3\}$ such that $v_1\treeordereq v_2$ and $w_2 \treeordereq w_3$. On the other hand, since the points $v_2$ and $w_2$ are in the same $N$-orbit, there exists $h\in N$ such that $h.w_2=v_2$. Thus, we have $v_1\treeordereq h.w_2\treeordereq h.w_3$ and, since $h.w_3$ projects to $x_3$, we get $x_1 \treeordereq x_3$, as desired. To check antisymmetry of $\treeorder$, suppose that for $x_1,x_2\in\widehat\Tbb_0$ it holds that $x_1\treeordereq x_2$ and $x_2\treeordereq x_1$. In this case, again by the definition of the partial order, there exist $v_i, w_i\in\pi^{-1}(x_i)$ for $i\in\{1,2\}$, such that $v_1\treeordereq v_2$ and $w_2\treeordereq w_1$. Take $h_1,h_2\in N$ so that $w_i=h_i.v_i$ for $i\in\{1,2\}$. Then, applying $h_2^{-1}$, we  get $v_1\treeordereq v_2\treeordereq h_2^{-1}h_1.v_1$. Since $h_2^{-1}h_1$ is elliptic, these must be equalities, implying that $v_1=v_2$ and therefore $x_1=x_2$. This finishes the proof that $(\widehat\Tbb_0,\treeorder)$ is a poset.
	
We proceed to check that $(\widehat\Tbb_0,\treeorder)$ satisfies condition \ref{i-tree} in the definition of pre-directed tree. Recall that for given $x\in \widehat{\Tbb}_0$,  we write $[x, \widehat{\infty_\Tbb}[\,=\{y\in \widehat{\Tbb}_0: x\treeordereq y\}$.
We want to show that $[x, \widehat\infty_\Tbb[$ is order-isomorphic to the half-line $[0, +\infty[$. To see this, take $v\in\pi^{-1}(x)$ and denote by $\pi_0\colon [v,\infty_\Tbb[\to\widehat\Tbb$ the restriction of the projection $\pi$ to the line $[v,\infty_\Tbb[$. In order to prove the  claim we will show that $\pi_0$ is an order-preserving bijection onto $[x,\widehat\infty_\Tbb[$. Note that, by construction, $\pi_0$ is order-preserving and takes values in $[x, \widehat{\infty_\Tbb}[$. 
From the assumption that the action of $N$ is horocyclic, we get that $\pi_0$ is injective. Let us prove that $\pi_0$ is surjective. For this, take $x'\in[x,\widehat\infty_\Tbb[$. Then, there exist $w'\in\pi^{-1}(x')$, $w\in\pi^{-1}(x)$ with $w\treeordereq w'$, and $h\in N$ with $h.w=v$. Since $v=h.w\treeordereq h.w'$, we deduce that  $h.w'\in[v,\infty_\Tbb[$, and $\pi_0(h.w')=\pi_0(w')=x'$, as desired. This proves condition \ref{i-tree}. Condition \ref{i-chains-R} in the definition of a pre-directed tree follows taking a countable subset $\Sigma_0\subset \Tbb$ and considering the subset $\widehat\Sigma_0:=\pi(\Sigma_0)$. Finally, for condition \ref{c-predirected}, take $x_1,x_2\in \widehat T_0$, and for $i\in \{1,2\}$ take $v_i\in \pi^{-1}(x_i)$. As $\Tbb$ is a directed tree, the least common upper bound $v_1\treeup v_2\in \Tbb$ exists, and set $y=\pi(v_1\treeup v_2)\in \Tbb_0$. By the definition of the partial order on $\Tbb_0$, we see that $y$ is a common upper bound for $x_1$ and $x_2$.
Thus, we have proved that $(\widehat\Tbb_0,\treeorder)$ is a pre-directed tree. Notice that the action $\Phi $ projects to an action on $\widehat\Tbb_0$ which clearly preserves $\treeorder$. Denote by $\widehat{\Phi}_0\colon G\to\Aut(\widehat{\Tbb}_0,\treeorder)$ this action and notice that $N$ is contained in its kernel. It follows from the proof that $(\widehat\Tbb_0,\treeorder)$ is a pre-directed tree, that the projection map $\pi\colon \Tbb\to\widehat\Tbb_0$ is a $G$-equivariant and surjective grading. Moreover, it follows directly from the definitions that in the case where $\Phi $ is d-minimal, so is $\widehat\Phi _0$.
	
Finally, by applying Lemma \ref{lem.predirectedtree} we get a directed tree $(\widehat\Tbb,\treeorder)$ together with an action $\widehat\Phi\colon G\to\Aut(\widehat\Tbb,\treeorder)$ which grades $\widehat\Phi_0$. Moreover, Lemma \ref{lem.predirectedtree} also implies that if $\widehat\Phi_0$ is d-minimal, so is $\widehat\Phi$. Notice that, since $N$ is contained in the kernel of $\widehat\Phi_0$ and gradings are surjective, $N$ is also contained in the kernel of $\widehat\Phi $. Finally, since the composition of gradings is a grading, the result follows.
\end{proof}

\begin{ex} \label{e-predirected}
Let us sketch an example of an action $\Phi\colon G\to\mathsf{Aut}(\mathbb{T,\treeorder})$ such that $G$ contains a normal subgroup $N$ whose action is horocyclic, and the induced quotient $\mathbb{T}/N$ is not a directed tree (but a pre-directed tree). This construction will be a variation of the construction of Plante-like actions in \S \ref{s-wreath}. For this, consider a finitely generated group $H$ and a minimal action $\varphi_0\colon H\to\homeo_0(\R)$, and let $\Omega\subset\R$ be the $\varphi_0$-orbit of $0$. We will define $G$ as a finitely generated subgroup of the \emph{unrestricted} wreath product $\Z \wrwr_{\Omega} H$. Recall that the latter is defined as the semi-direct product $\Z \wrwr_\Omega H:=\Z^\Omega\rtimes H$, where $\Z^\Omega$ is the set of all functions from $\Omega$ to $\Z$ (with no restriction on the support), viewed as an abelian group with pointwise addition, and $H$  acts on $\Z^\Omega$ by the shift action $h\cdot f(x)=f(h^{-1}.x)$.  To define $G$, consider  a decreasing sequence $(x_n)_{n\in\N}$ contained in $\Omega$ and converging to $0$, and denote its image by $S$. Set 
\[G=\langle H,\delta_0,\delta_S\rangle\le \Z \wrwr_\Omega H,\]
 where $\delta_0$ and $\delta_S$ are the indicator functions of the sets $\{0\}$ and $S$ respectively. Notice that  ${L}:=\Z^\Omega\cap G$ is a countable subgroup of $\Z^\Omega$ (generated by the $H$-translates of $\delta_S$ and $\delta_0$), and $G=L\rtimes H$. Notice also that the $H$-translates of $\delta_0$ generate the subgroup of finitely supported functions $L_0:=\bigoplus_\Omega \Z \le L$, which is normal in $G$.  Finally observe that by construction,  the support of the function $\delta_S$ is bounded above in $\Omega$ (with respect to the order induced by $\R$), and thus the same is true for every $f\in {L}$. Given distinct functions $f, h$ in $L$,  denote by $m_{f, h}\in \R$ the supremum of $\mathsf{Supp}(f-h)$, i.e.\ the smallest point of $\R$ such that $f|_{(m_{f, h}, +\infty)\cap \Omega}=h|_{(m_{f, h}, +\infty)\cap \Omega}$.
 
Set $\mathbb{T}=(L\times \R)/_\sim$, where $\sim$ is the equivalence relation defined by
\[(f_1,x_1)\sim(f_2,x_2)\Leftrightarrow\Big(x_1=x_2\text{ and }f_1(z)=f_2(z)\text{ for every }z>x_1\Big).\]
Denote by $[f, x]$ the equivalence class of $(f, x)$, and endow $\Tbb$ with the partial order $\treeorder$ defined by $[f_1,x_1]\treeorder[f_2,x_2]$ if and only if $x_1<x_2$ and $f_1|_{(x_2, +\infty)\cap \Omega}=f_2|_{(x_2, +\infty)\cap \Omega}$. Equivalently, the poset $(\Tbb, \treeorder)$ can be thought as follows: for each $f\in L$, take an isomorphic copy of $(\R, <)$ (corresponding to the subset $\{f\}\times \R$ of $L\times \R$), and glue the copies corresponding to $f\neq h$ along the ray $[m_{f, h}, +\infty)$. With this picture in mind, it is easy to verify that $(\Tbb, \treeorder)$ is a directed tree.  For condition \ref{i-directed}, notice that the smallest upper bound of $[f,x]$ and $[h,y]$ is $[f,\max\{x,y,m_{f, h}\}]$ (recall also that $L$ is countable to verify \ref{i-chains-R}).

Consider the diagonal action of $G$ on $L\times \R$, where the action on the ${L}$-coordinate is the natural action associated with the semi-direct product decomposition $G=L\rtimes H$ (i.e.\ $L$ acts on itself by translations, and $H$ acts on $L$ by the shift), while  the action on the $\R$-coordinate is obtained by projecting $G$ onto the factor $H$ and then acting via $\varphi_0$. It is direct to check that this diagonal action preserves the equivalence relation $\sim$, and descends to an order-preserving action $\Phi\colon G\to\Aut(\Tbb,\treeorder)$. One can also show that $\Phi$ is d-minimal (this essentially follows from minimality of $\varphi_0$ combined with the transitivity of the action of $L$ on itself). The subgroup ${L}$ is horocyclic for $\Phi$, and $\Tbb/L$ is a directed tree isomorphic to $(\R, <)$, on which $G$ acts via $\varphi_0$ (in particular, $\Phi$ is horograded by $\varphi_0$). In contrast, consider the subgroup $L_0$ of finitely supported functions, which is also horocyclic for $\Phi$ (as it is contained in $L$). We claim that $\Tbb/L_0$ is not a directed tree (though a pre-directed tree, see the proof of Proposition \ref{l-quotienting-horocyclic}). For this,  denote by $f_0$ the constant function $f_0\equiv 0$. We claim that the projections of the points $u:=[f_0,0]$ and $v:=[\delta_S,0]$ to $\Tbb/L_0$ have no least common upper bound. For this notice that, by construction of the set $S$, the function $\delta_S|_{(x, +\infty)\cap \Omega}$ has finite support for every $x>0$. It follows that for every $x>0$ there exists $h\in L_0$ such that $h.[f_0,x]=[\delta_S,x]$, and therefore  every point strictly above $u$ (or $v$) projects to a common upper bound for $u$ and $v$ in $\Tbb/L_0$. On the other hand, the points $u$ and $v$ have distinct projections, since the set $\mathsf{Supp}(\delta_S)\cap(0,+\infty)$ is infinite. This finishes the proof of our claim. 
\end{ex}

\section{Finding horogradings for virtually solvable groups}\label{s.Rfocal-solvable}
\label{sec Focal solvable}

\subsection{Actions of solvable groups on directed trees}

The results in the previous sections imply the following for actions of solvable groups on directed trees. Recall that given a virtually solvable group, we denote by $\Fit(G)$ its Fitting subgroup and by $\vf(G)$ its virtual Fitting length (see \S \ref{s-fitting}). 

\begin{thm}\label{t-solvable-inductive}
Any d-minimal action $\Phi\colon G\to \Aut(\Tbb, \treeorder)$ of a virtually solvable group on a directed tree $\Tbb\neq \R$, can be horograded by a minimal action $j\colon G\to \homeo_0(\R)$ such that $\Fit(G)\subset \ker j$.
\end{thm} 

\begin{proof} Since $\Fit(G)$ is the union of the nilpotent normal subgroups, and the action of a nilpotent normal subgroup is either trivial or horocyclic by Corollary \ref{c-normal-focal-centralizer}, we deduce that the action of $\Fit(G)$ must be either trivial or horocyclic. Thus, applying Proposition \ref{l-quotienting-horocyclic} to $N=\Fit(G)$, we obtain that $\Phi$ is graded by a d-minimal action $\Phi_0\colon G\to \Aut(\Tbb_0, \treeorder)$ where $\Fit(G)$ acts trivially. Since $\Phi_0$ is d-minimal, $G/\Fit(G)$ must be infinite and therefore $\vf(G/\Fit(G))=\vf(G)-1\geq 1$. In particular $\vf(G)\geq 2$. The proof proceeds by induction on $\vf(G)$. Assume that $\vf(G)=2$, and consider the d-minimal action $\Phi_0\colon G\to\Aut(\Tbb_0,\treeorder)$ with $\Fit(G)\subset\ker \Phi _0$ grading $\Phi $. Since $\vf(G/\Fit(G))=1$, it must hold that $\Tbb_0=\R$ and therefore the conclusion follows by setting $j=\Phi_0$. For the inductive step, consider a d-minimal action $\Phi\colon G\to\Aut(\Tbb,\treeorder)$ and  its corresponding grading action $\Phi _0\colon G\to\Aut(\Tbb_0,\treeorder)$. We distinguish two cases according to whether $\Tbb_0=\R$ or not. In the former case the conclusion follows by setting $j=\Phi _0$. In the latter case, we can apply the induction hypothesis to $\Phi _0$, since it factors through $G/\Fit(G)$ which has strictly smaller virtual Fitting length. Therefore $\Phi _0$ is horograded by an action $j\colon G\to \homeo_0(\R)$ such that $\Fit(G)\subset \ker j$ and the same holds for $\Phi$.
\end{proof}

\subsection{Actions of virtually solvable groups on the line} 

We are now ready to prove Theorem \ref{t-main} from the introduction. Recall that it states that any minimal action $\varphi\colon G\to \homeo_0(\R)$ of a finitely generated virtually solvable group is either conjugate to an affine action, or laminar and horograded by a minimal or cyclic action $j\colon G\to \homeo_0(\R)$ factoring through $G/\Fit(G)$.
\begin{proof}[Proof of Theorem \ref{t-main}]
Let $\varphi$ be a minimal action of $G$. It was already shown in \cite[Theorem 8.3.8]{BMRT}  that $\varphi$  is either laminar or
conjugate to an affine action  ($G$ is supposed there to be solvable but the proof works with minimal changes when $G$ is virtually solvable). Let us reproduce a sketch for completeness, and refer to  \cite[Theorem 8.3.8]{BMRT} for more details.  Upon  replacing $G$ by a quotient (which remains virtually solvable), we can suppose that $\varphi$ is faithful. Let $A$ be an abelian normal subgroup of $G$. Then $\fix^{\varphi}(A)=\varnothing$, by minimality. We distinguish two cases according to whether $A$ acts freely or not. If $A$ acts freely, then by H\"older's theorem  its action is semi-conjugate to an action by translations associated with an embedding $\iota \colon A\to (\R, +)$ (see for instance Navas \cite[\S 2.2.4]{Navas-book}).  If $\iota(A)$ is dense in $(\R, +)$,  then $\varphi(A)$ has a unique minimal invariant subset $\Lambda \subseteq \R$.  As $A$ is normal, $\Lambda$ is preserved by $\varphi$, and thus by minimality $\Lambda=\R$ and $\varphi(A)$ is actually conjugate to $\iota(A)$; after conjugating, we can assume $\varphi(A)=\iota(A)$. Then the Lebesgue measure is the unique $A$-invariant Radon measure up to a positive constant. It follows that every element of  $\varphi(G)$ sends the Lebesgue measure to a multiple of itself, and thus is an affine map. In the case where $\iota(A)$ is cyclic, we have that $\varphi(A)$ is conjugate to a cyclic group of translations centralized by $\varphi(G)$, and $\varphi$ descends to a minimal action on the topological circle $\R/\varphi(A)$. Since $G$ is amenable, this action preserves a probability measure, and thus $\varphi$ preserves a non-zero Radon measure on $\R$. In this case it follows that $\varphi$ is conjugate to an action by translations, thus also affine. 
Assume now that $A$ acts non-freely;  then there exists $a\in A\setminus\{\id\}$ such that $\fix^{\varphi}(a)\neq \varnothing$. But $\fix^{\varphi}(a)$ is an $A$-invariant closed subset (as all elements of $A$ commute with $a$), and thus it must accumulate on both $\pm \infty$. Let $I$ be a (bounded) connected component of $\suppphi(a)$. Then for every $g\in G$, $g.I$ is a connected component of $\suppphi(gag^{-1})$, and since $gag^{-1}\in A$ commutes with $a$, we have that $g.I$ and $I$ cannot cross. Thus the closure of the orbit of $I$ in $\R^{(2)}$ is a $\varphi$-invariant lamination $\mathcal{L}$, and since $\varphi$ is minimal, $\mathcal{L}$ is automatically a covering lamination, and $\varphi$ is focal with respect to $\mathcal{L}$ (Remark \ref{r-minimal-implies-focal}). This shows that $\varphi$ is either laminar or conjugate to an affine action.

 Now, in the laminar case, by Proposition \ref{prop.minimalamination} we can choose a minimal $\varphi$-invariant lamination $\mathcal{L}$. We then apply  Proposition \ref{prop.minimalmodeltree} to  find an action $\Phi \colon G\to \Aut(\Tbb, \treeorder)$ on a directed tree, which is either d-minimal or simplicial and transitive on branching points, and an equivariant map $\iota\colon (\mathcal{L}, \subseteq)\to (\Tbb, \treeorder)$. If the action $\Phi $  is d-minimal, then by Theorem \ref{t-solvable-inductive}, it can be horograded by an action $j\colon G\to \homeo_0(\R)$ which satisfies the desired conclusion, and thus $\varphi$ is also horograded by $j$ (Remark \ref{r-tree-horograding}). If it is simplicial, we have from Proposition \ref{prop.Rtreemetric} that $\Phi$ can be horograded by a cyclic action $j\colon G\to \homeo_0(\R)$. Then every element $g\in G$ with $g\notin \ker j$ is a hyperbolic isometry. As argued at the beginning of the proof of Theorem \ref{t-solvable-inductive}, Corollary \ref{c-normal-focal-centralizer} gives that $\Phi(\Fit(G))$ is horocyclic, so that $\Fit(G)\subset \ker j$.
\end{proof}

A crucial aspect of Theorem \ref{t-main} is that if $\varphi$ is laminar, then we may apply  the theorem again to the  action  $j$ which horogrades $\varphi$. This opens the way to reason inductively to study actions of solvable groups on the line.
As a first simple illustration of this philosophy we prove Theorem \ref{mthm.affine_trace}, which is its most direct consequence. Indeed the combination of Theorem \ref{t-main} and Proposition \ref{p-horograding-germ} gives the following.

\begin{prop}\label{p-germs-solvable}
Let $\varphi \colon G \to \homeo_0(\R)$ be a minimal laminar action of a  finitely generated virtually solvable group. Then, there exist an irreducible affine action $\psi\colon G\to \Aff(\R)$ and two maps $h_+, h_-\colon \R\to \R$
satisfying the following conditions:
\begin{enumerate}
	\item $h_+$ is monotone non-decreasing,
	\item $h_-$ is monotone non-increasing,
	\item we have $\lim_{x\to +\infty}h_+(x)=\lim_{x\to -\infty} h_-(x)=+\infty$,
	\item for every $g
	\in G$ there exists $M>0$ such that $ h_+(\varphi(g)(x))=\psi(g) ( h_+(x))$ for every $x>M$, and $ h_-(\varphi(g)(x))=\psi(g)( h_-(x))$ for every $x<-M$. 
\end{enumerate}
\end{prop}

\begin{proof}
	We only discuss the existence of $h_+$, the case of $h_-$ being analogous.
We argue by induction on the virtual Fitting length $\vf(G)$ of $G$. Let $\varphi\colon G\to \homeo_0(\R)$ be a minimal laminar action. Apply Theorem \ref{t-main} to obtain an action $j\colon G\to \homeo_0(\R)$ as in the statement that horogrades $\varphi$. We apply Proposition \ref{p-horograding-germ} to obtain a map $h'_+$ as in its statement. If $j$ is an affine action (which includes the possibility that it is cyclic), then we have already reached the conclusion.  Else, $j$ is laminar, and by the inductive hypothesis for $G/\Fit(G)$, there exist an irreducible affine action $\psi$ and a non-decreasing map $h''_+\colon \R\to \R$ satisfying $\lim_{x\to +\infty} h''_+(x)=+\infty$ and $h''_+( j(g)(x))=\psi(g) ( h''_+(x))$ for any sufficiently large $x\in \R$. We then set $h_+=h''_+\circ h'_+$.
\end{proof}

We are now ready to give the proof of Theorem \ref{mthm.affine_trace}.  Recall that it states if $G$ is a finitely generated virtually solvable group, and $\varphi\colon G\to \homeo_0(\R)$ is an  action without global fixed points, then there exists an irreducible {\em affine} action $\psi\colon G\to \Aff(\R)$,  an interval $I$ of the form $(a,+\infty)$, and a non-decreasing map $h\colon I\to \R$, with $\lim_{x\to +\infty} h(x)=+\infty$  such that for every $g\in G$ we have	\[\psi(g)(h(x))= h(\varphi(g)(x))\]	for  all $x\in \R$ sufficiently large. Moreover, this $\psi$ is unique up to affine conjugacy.

\begin{proof}[Proof of Theorem \ref{mthm.affine_trace}] Let $\varphi\colon G\to \homeo_0(\R)$  be an action of a finitely generated virtually solvable group $G$ on the line.	If the action $\varphi\colon G\to \homeo_0(\R)$ is semi-conjugate to an affine action, we simply have to take this action as the desired $\psi$. Otherwise, by Theorem \ref{t-main}, we have that $\varphi$ is semi-conjugate to a minimal laminar action, so that we can apply Proposition \ref{p-germs-solvable}, which gives the desired irreducible affine action $\psi$.
	
	We next prove that $\psi$ is unique up to affine conjugacy. 
	For this, assume there exist two such irreducible affine actions $\psi_1$ and $\psi_2$, with corresponding maps $h_1\colon I_1\to \R$ and $h_2\colon I_2\to \R$.  	Then, for  any $g\in G$ and sufficiently large $x\in \R$, we have that
	\[
	\psi_2(g)h_2h_1^{-1}(x)=h_2\varphi(g)h_1^{-1}(x)=h_2h_1^{-1}\psi_1(g)(x),
	\]
	so that the actions $\psi_1$ and $\psi_2$ are ``semi-conjugate near $+\infty$'' by the map $h:= h_2\circ h_1^{-1}$, defined on some interval of the form $(c,+\infty)$. 
	For $i\in \{1,2\}$ we write $A_i=\psi_i(G)$. Observe first that if $\psi_1(g)$ is trivial for some $g$, then $\psi_2(g)$ is an affine map fixing the image of $h$, and thus is trivial as well. Thus $\ker \psi_1=\ker \psi_2$. 
It follows that if  one of the groups $A_i$ is a cyclic group (necessarily generated by a translation), then the other is as well, and this case is readily solved as any two positive translations are conjugate by an affine map. We can thus assume that $A_i$ is not a cyclic group of translations (and thus acts minimally on $\R$). 

Given a finite subset $S$ of homeomorphisms of the line and a point $x\in \R$, denote by $\mathcal{R}(S, x)$ the equivalence relation on the interval $(x, +\infty)$  defined by $(y, z)\in \mathcal{R}(S, x)$ if there exist $s_1,\ldots, s_n\in S$ such that $s_n\cdots s_1(y)=z$ and $s_i\cdots s_1(y)>x$ for every $i\in \{1,\ldots, n\}$. 

\begin{claim}
For $i\in \{1, 2\}$, there exists $S_i\subset A_i$ such that for every sufficiently large $x\in \R$, the relation $\mathcal{R}(S_i, x)$ is minimal (i.e.\ its equivalence classes are dense in $(x, +\infty)$). 
\end{claim}
\begin{proof}[Proof of claim]
Assume first that $A_i$ is contained in the group of translations. Since it is finitely generated and not cyclic, it contains two rationally independent  translations $a, b$, and one readily checks that $\mathcal{R}(\{a^{\pm1 }, b^{\pm 1}\}, x)$ is minimal for every $x\in \R$.  Otherwise $A_i$ contains a contracting homothety $a$ and a positive translation $b$. Up to conjugating by an affine map, assume that $a(y)=\lambda y$ with $\lambda<1$ and $b(y)=y+1$. For $x>0$ and $n, m>0$ we have $a^m b^na^{-m}(y)=y+n\lambda^{m}$ and one readily checks that the successive applications of $a^{\pm 1}, b$  in the word $a^m b^na^{-m}$ never move $y$ to the left of itself. Since $\lambda^m$ can be made arbitrarily small, we deduce that  $\mathcal{R}(\{a^{\pm 1}, b\}, x)$ is minimal for every $x>0$. 
\end{proof}

Let now $\tilde{S}_i\subset G$ be finite subsets such that $\psi_i(\tilde{S}_i)=S_i$. Choose $x$ large enough so that $\mathcal{R}(S_1, x)$ and $\mathcal{R}(S_2, h(x))$ are both minimal, and such that the equivariance condition $h\circ \psi_1(s)(y)=\psi_2(s)\circ h(y)$ holds for all $y\ge x$ and all $s\in \tilde{S}_1\cup \tilde{S_2}$. Then the image of $h|_{(x, +\infty)}$ is invariant under the relation $\mathcal{R}(S_2, h(x))$. Thus, arguments similar to those  in Lemma \ref{l-dynamical-realisation} show that this image must be dense, implying that $h$ is continuous on $(x, +\infty)$. Again as in the proof of Lemma \ref{l-dynamical-realisation}, minimality of $\mathcal{R}(S_1, x)$ shows that $h|_{(x, +\infty)}$ must be injective and thus a homeomorphism onto its image.

Now note that since $K:=\ker \psi_1=\ker \psi_2$, we have that $A_1$ is abelian (and hence a group of translations) if and only if $A_2$ is. Otherwise, both groups are non-abelian and their derived subgroups consist of translations. In the first case, set $H=G$ and in the  second case set $H=[G, G]$. In either case $\psi_i(H)$ is a minimal group of translations for $i\in\{1, 2\}$.
Let $g\in H$ be such that $\psi_1(g)$ is a non-trivial positive translation. Then so  is $\psi_2(g)$, by the equivariance near $+\infty$ of the map $h$. Up to conjugating $\psi_2$ by an affine map we can assume $\psi_1(g)=\psi_2(g)$. Let $J_1\subset (x,+\infty)$ be a fundamental interval for the translation $\psi_1(g)$ and set $J_2=h(J_1)$, which is also a fundamental interval for the translation $\psi_1(g)=\psi_2(g)$. Therefore, up to conjugating the action $\psi_2$ by a translation, we can assume that $J_1=J_2$. Then for any $i\in \{1,2\}$, $\psi_i|_H$ descends to a minimal action by rotations on the circle $J_i/\psi_i(g)$. Moreover, the homeomorphism $h\colon (c,+\infty)\to (h(c),+\infty)$ induces a homeomorphism $\overline h$ of $J_1/\psi_1(g)=J_2/\psi_2(g)$ which conjugates the two actions. This implies that $\overline h$ is a rotation, and thus $h$ is a translation. As $h(J_1)=J_1$, we deduce that $h$ is trivial. It follows that for every $g\in G$, the affine maps $\psi_1(g)$ and $\psi_2(g)$ coincide on a neighborhood of $+\infty$, and thus are equal.
\end{proof}

For later use, let us record the following consequence of Proposition \ref{p-germs-solvable}.
\begin{lem} \label{l-pseudohomtothety-exist}
Let  $\varphi\colon G\to \homeo_0(\R)$ be a minimal laminar action of a finitely generated virtually solvable group. Then there exists an element $h\in G$ whose image is a pseudo-homothety.  
\end{lem}
\begin{proof}
Let $\psi\colon G\to \Aff(\R)$ be the affine action given by Proposition \ref{p-germs-solvable}. Choose $g\in G$ with $\psi(g)\neq \id$. Assume, say, that  $\psi(g)(x)>x$ for every sufficiently large $x$. Then $\varphi(g)(x)>x$ for all $x$ close enough to $+\infty$, and similarly $\varphi(g)(x)<x$ for all $x$ close enough to $-\infty$. Thus $\varphi(g)$ is a pseudo-homothety.
\end{proof}

\section{The case of $\Z\wr \Z$}\label{ssc.Plante}

The goal of this  section is to apply the previous discussion to the lamplighter group $\Z \wr \Z$. This gives an  illustration of how Theorem \ref{t-main} can be applied in the simplest special case, and at the same time the results of this section will be invoked in the proof of Theorem \ref{mthm.C1}. 

\subsection{Classification of actions} The following result (together with Proposition \ref{l-wr-affine}) classifies actions of $\Z \wr \Z$ on the line up to semi-conjugacy. Recall that we have constructed what we called Plante-like actions of $\Z\wr\Z$ in \S \ref{ssc.plante}.

\begin{prop}\label{p-ZwrZ}
Every irreducible action $\varphi\colon \Z\wr \Z \to \homeo_0(\R)$ is semi-conjugate either to an affine action or to a Plante-like action.
\end{prop}

\begin{proof}
Using notation from  \S\ref{sec ejemplo lamp lighter}, we set $G=\Z\wr \Z=L\rtimes\Z$ and consider the generators  $h_0\in L$ and $g$ (see \eqref{e-presentation-ZwrZ}). Recall also that we set $h_n=g^nh_0g^{-n} $ ($n\in \Z$).
By Theorem \ref{t-main}, it is enough to show that every minimal laminar action $\varphi\colon G\to \homeo_0(\R)$ is conjugate to one of the four Plante-like actions. So assume that $\varphi$ is minimal and laminar. By Theorem \ref{t-main}, $\varphi$ can be horograded by an action $j\colon G\to\homeo_0(\R)$ which is either cyclic or minimal, and where $\Fit(G)$ acts trivially. Since $\Fit(G)=L$ and $G/L\cong\Z$, we deduce that $j$ is a cyclic action and so $j(g)$ has no fixed points.  In particular, from Proposition \ref{p-classification-elements}, we get that the elements $\varphi(h_n)$ ($n\in \Z$) are totally bounded and that  $\varphi(g)$ is a homothety. Let us call $p\in \R$ the unique fixed point of $\varphi(g)$. Note that no element $\varphi( h_n)$ ($n\in \Z$) can fix $p$, as otherwise $p$ would be a global fixed point for the action.

Assume first that $\varphi(g)$ is an expanding homothety and that  $h_0.p>p$. It follows that  $h_n.p>p$ for every $n\in \Z$.   For fixed $n\in \Z$, let $I_n$ be the connected component of $\suppphi(h_n)$ containing $p$. Then, for every $n\in \Z$, the homothety $\varphi(g)$ sends $I_n$ to $I_{n+1}$ and thus $\overline{I_n}\subset I_{n+1}$; moreover for each $l<n$ the element $h_n$ must send $I_l$ entirely to the right of itself, as it must map it to a distinct connected component of $\suppphi(h_l)$. From this, it follows that for every element $f=h_{n_1}^{k_1}\cdots h_{n_r}^{k_r}$ with $n_1<\cdots<n_r$ and $k_r\neq 0$, we have $f.p>p$ if $k_r>0$, and $f.p<p$ otherwise. Now note that such an element $f$ corresponds to the polynomial $P(X)=k_1X^{n_1}+\cdots +k_r X^{n_r}$ under the isomorphism $L\cong \Z[X, X^{-1}]$, and the sign of $k_r$ also determines whether $P$ is positive or negative in the lexicographic order $\prec_{\max}^+$. Since $g.p=p$, it follows that the map that to any element $f\in L\cong \Z[X,X^{-1}]$ associates  $f.p$ is a $G$-equivariant order-preserving embedding of $(\Z[X,X^{-1}], \prec_{\max}^+)$ into $\R$, and thus, by Lemma \ref{l-dynamical-realisation}, $\varphi$ must be conjugate to the Plante-like action associated with $\prec_{\max}^+$ . 

When $\varphi(g)$ is an expanding homothety, but $h_0.p<p$ we similarly obtain that $\varphi$ is conjugate to the Plante-like action associated with $\prec_{\max}^-$. When $\varphi(g)$ is a contracting homothety, we get a conjugacy to the Plante-like action corresponding to $\prec_{\min}^+$ or $\prec_{\min}^-$ depending on which side the element $\varphi(h_0)$ moves $p$ to.
\end{proof}
 
\subsection{Plante-like subactions in laminar actions} 

 Plante-like actions of $\Z \wr \Z$ are in a precise sense the ``smallest'' laminar actions for solvable groups. This will be an important step in the proof of Theorem \ref{mthm.C1}.
 \begin{prop} \label{p-plante-ubiquitous}
 	Let  $\varphi\colon G\to \homeo_0(\R)$ be a faithful minimal laminar action of a finitely generated virtually solvable group. Then $G$ contains a subgroup isomorphic to $\Z \wr \Z$ whose action is irreducible and semi-conjugate to a Plante-like action.
 \end{prop}

 \begin{proof} 
 	By Lemma \ref{l-pseudohomtothety-exist} there exists an element $s\in G$ which acts as an expanding pseudo-homothety.  Let $[\xi_1, \xi_2]$ be an interval containing $\fix^{\varphi}(s)$.  Let $A$ be a non-trivial abelian normal subgroup of $G$. Since the action of $G$ is faithful and minimal, we have $\fix^\varphi(A)=\varnothing$, hence there exists $a\in A$ such that $a.\xi_0>\xi_1$. Then $\fix^\varphi(a)\cap \fix^\varphi(s)=\varnothing$ and thus $H:=\langle a, s\rangle$ acts irreducibly. The conjugates of $a$ by powers of $s$ all belong to $A$ and thus commute. It follows that there is a well-defined homomorphism $\rho\colon \Z\wr \Z\to H$  sending $h_0$ to $a$ and $g$ to $s$. Now $\varphi\circ \rho$ cannot be semi-conjugate to any affine action (for instance because $\varphi(G)$ does not contain any element conjugate to a translation,  by Proposition \ref{p-classification-elements}), and thus it is semi-conjugate to a Plante-like action by Proposition \ref{p-ZwrZ}; we have \textit{a fortiori} that $\rho$ is an isomorphism, as Plante-like actions are faithful. \qedhere
 \end{proof}

\section{$C^1$ actions on intervals} \label{sc.C1}
 In this section we prove Theorem \ref{mthm.C1}. After Proposition \ref{p-plante-ubiquitous}, we are reduced to consider actions which are semi-conjugate to Plante-like actions of $\Z \wr \Z$. Let us outline the main idea, which comes from the work of Bonatti, Monteverde, Navas, and the third author \cite{BMNR}.  It is proven there that for any $C^1$ action of the Bausmlag--Solitar group $\BS(1, n)=\langle a, b\mid aba^{-1}=b^n\rangle$ on the closed interval, which is topologically conjugate to its standard affine action, the derivative of the image of $a$ at its unique fixed point must be equal to $n$. Observe that the affine action of $\BS(1, n)$ can also be seen as an action of $\Z \wr \Z$, by precomposing it with the epimorphism $\Z\wr \Z\to \BS(1, n)$ defined on the generators as $h_0\mapsto b$ and $g\mapsto a$. Letting $n$ increase we obtain a sequence $(\varphi_n)$ of actions of $\Z \wr \Z$. One can show (using Proposition \ref{p-ZwrZ}) that  the sequence $(\varphi_n)$ can be conjugated to make it converge to a Plante-like action of $\Z \wr \Z$\footnote{The intuition behind this claim is that the lexicographic orders on the ring of Laurent polynomials $\Z[X, X^{-1}]$, which gives rise to the Plante-like actions, can be approximated by a one-parameter family of preorders obtained by setting $P\ge_\lambda 0$ if $P(\lambda)\ge 0$, for $\lambda>0$, as $\lambda$ approaches $+\infty$. These preorders correspond to the affine actions of $\Z \wr \Z$. We omit details as these facts will not be formally needed.}. This strongly suggests that if a Plante-like action of $\Z \wr \Z$ were conjugate to a $C^1$ action, the derivative of $g$ at its unique fixed point should be infinite, thus providing a contradiction. While this limit argument is not an actual proof (as the conjugacy and the convergence are only $C^0$), it turns out that the argument in the proof of \cite[Proposition 4.13]{BMNR} can be adapted to Plante-like actions to obtain the following.

\begin{prop}\label{p-plante-C1}
Let $\rho\colon H\to \Z\wr \Z$ be an epimorphism from a finitely generated group $H$, and let $\psi \colon H\to \homeo_0(\R)$ be the action obtained by postcomposing $\rho$ with a Plante-like action. Let $\varphi\colon H\to \homeo_0([0, 1])$ be an action semi-conjugate to $\psi$. Then, $\varphi$ cannot be of class $C^1$ on $[0,1]$.  
\end{prop}

Before the proof, let us recall from  \S\S \ref{ssc.plante} and  \ref{s-plante-laminations} some general properties of Plante-like actions, that will be used in the proof. 
As before, we denote by $g$ and $h_0$ the standard generators of $\Z \wr \Z$, and set $h_n=g^n h_0g^{-n}$. For definiteness, let  $\eta\colon \Z\wr\Z\to \homeo_0(\R)$ be the Plante-like action associated with the lexicographic order $\prec_{\max}^{+}$. Recall that we choose the order-preserving bijection $t\colon \Z[X, X^{-1}]\to \Q$ used to construct the Plante-like action in such a way that the polynomial 0 is sent to the origin $0\in \R$, and this implies that $\eta(g)$ is an expanding homothety with $0$ as fixed point. In contrast, for every $n\in \Z$,  the element $h_n$ satisfies $\eta(h_n)(x)\ge x$ for every $x\in \R$, with strict inequality for $x=0$ (since $h_n$ corresponds to the  polynomial $X^n\in \Z[X, X^{-1}]$, which is positive for the lexicographic order $\prec_{\max}^+$).

For $n\in \Z$, let $I_n$ be the connected component of $\supp^\eta(h_n)$ containing $0$. Recall from the proof of Proposition \ref{prop Plante is laminar} (in particular, see Claim \ref{l-plante-fixed-points}) that each $I_n$ is a bounded interval belonging to an invariant lamination, and that $\overline{I}_n\subseteq I_{n+1}$. Note also that $\eta(g)(I_n)=I_{n+1}$ (since $gh_ng^{-1}=h_{n+1}$). We will use the following  property of the Plante-like action $\eta$.
\begin{lem} \label{l-disjoint-intervals}
With notation as above, for $n\in \Z$, consider the subgroup  $L_n:=\langle h_j\, (j\geq n)\rangle$ of $\Z\wr \Z$. 
If $f_1, f_2\in L_{-n}$ are distinct elements, then   $\eta(f_1)(I_{-n-1})$ and $\eta(f_2)(I_{-n-1})$ are disjoint. \end{lem}

\begin{proof}
It is enough to show that $\eta(f)(I_{-n-1})$ and$I_{-n-1}$  are disjoint whenever $f\in L_{-n}$ is non-trivial. Such an $f$ can be uniquely written as $f=h_{m_1}^{k_1}\cdots k_{m_r}^{k_r}$, with $m_1>\cdots >m_r\ge -n$ and $k_1,\ldots, k_r\neq 0$. Claim \ref{l-plante-fixed-points} in Proposition \ref{prop Plante is laminar} gives the chain of inclusions $I_{-n-1}\subset I_{m_1}\subset \cdots \subset I_{m_r}$. Since each $I_{m_i}$ is $\eta(h_{m_i})$-invariant,  this implies
\[\eta(h_{m_2}^{k_2}\cdots h_{m_r}^{k_r})(I_{-n-1})\subset I_{m_2}\subseteq I_{m_1-1}.\]
Now note that $\eta(h_{m_1})^{k_1}$ moves both endpoints of $I_{m_1-1}$, and in the same direction (depending on the sign of $k_1$); since $\eta(h_{m_1}^{k_1})(I_{m_1-1})$ and $I_{m_1-1}$ cannot cross, they must be disjoint. This gives the desired conclusion. 
\end{proof}
 
Lemma \ref{l-disjoint-intervals} allows to produce plenty of disjoint images of $I_0$ using controlled group elements. On the other hand, if the action $\eta$ were semi-conjugate to a $C^1$ action on $[0, 1]$, the length of these intervals can be bounded below using derivatives. This is the strategy to obtain a contradiction in the proof below.

\begin{proof}[Proof of Proposition \ref{p-plante-C1}] 
As above we assume that $\eta\colon \Z\wr\Z\to \homeo_0(\R)$ is the Plante-like action associated with the lexicographic order $\prec_{\max}^{+}$, and that the action $\psi$ in the statement is $\psi=\eta\circ \rho$ (the case of the other Plante-like actions can be treated similarly). We choose and fix preimages $\widetilde{g}, \widetilde{h}_0\in H$ of the generators $g, h_0$ under the epimorphism $\rho$, and set $\widetilde{h}_n:=\widetilde{h}_0$.   Let $\varphi$ be as in the statement and, looking for a contradiction, assume  that $\varphi$ is of class $C^1$. As $H$ is finitely generated, using a trick attributed to Muller \cite{Muller} and Tsuboi \cite{Tsuboi} (see also \cite{BMNR}), we can assume that $\varphi(k)'(0)=\varphi(k)'(1)=1$ for every $k\in H$.

Let $\tau\colon(0, 1)\to \R$ be a non-decreasing map that semi-conjugates $\varphi$ to $\psi$, in the sense that $\tau\varphi(k)=\psi(k)\tau$ for every $k\in H$. As Plante-like actions are minimal (Proposition \ref{prop Plante-like minimal}), the semi-conjugacy $\tau$ must be  continuous. To avoid confusion, we will use Latin letters to denote points of $\R$ (where the Plante-like action $\eta$ is defined), and Greek letters for points of $[0, 1]$ (where the action $\varphi$ is defined). Set $J_n=\tau^{-1}(I_n)$ for $n\in \Z$. Note that $\varphi(\tilde{g})(J_n)=J_{n+1}$ for every $n\in \Z$. Let $\lambda:=\min_{\xi\in[0,1]}|D\varphi(\widetilde{g}^{-1})(\xi)|$. The mean value theorem implies that for every $n\in \Z$ we have the length estimate  
\begin{equation} |J_{-n}|\geq |J_0|\lambda^{n}. \label{e-Jn}\end{equation}
We fix $N>\frac{1}{\lambda}$ and $\varepsilon>0$  such that 
\begin{equation} (1-\varepsilon)^{2N+2}N\lambda>1. \label{e-epsilon-choice} \end{equation}
Let $\sigma\in (0,1)$ be such that  for any $\xi\in (\sigma, 1]$ and  $s\in  \{\widetilde{h}_0, \widetilde{g}, \widetilde{g}^{-1} \}$, one has $ D\varphi(s)(\xi)>1-\varepsilon$.  Since the subgroup $L$ acts without fixed points in the Plante-like action, we can find an element $a\in L$ such that $\eta(a)(\overline{I_0})\subset (\tau(\sigma), 1]$. Hence, choosing a preimage $\widetilde{a}\in \rho^{-1}(a)$ we have $\varphi(\widetilde{a})(\overline{J_0})\subset (\sigma, 1]$.

Fix $n\geq 1$.  Given  an $n$-tuple  of integers $\underline{i}=(i_0, \ldots, i_{n-1})\in \{1,\ldots,N\}^n$, set $\widetilde{f}_{\underline{i}}=\widetilde{h}_{-n+1}^{i_{n-1}}\cdots \widetilde{h}_{-1}^{i_1}  \widetilde{h}_0^{i_0}$ and ${f}_{\underline{i}}=\rho\left (\widetilde{f}_{\underline{i}}\right )$. Note that the intervals of the form $\varphi(\widetilde{f}_{\underline{i}}\widetilde{a})(J_{-n})$ are  pairwise disjoint when $\underline{i}$ varies. Indeed, the semi-conjugacy  $\tau$ maps each such interval to the corresponding interval $\eta(f_{\underline{i}}a)(I_{-n})=\eta(af_{\underline{i}})(I_{-n})$, and these are pairwise disjoint by Lemma \ref{l-disjoint-intervals}, since the elements $f_{\underline{i}}$ all belong to the subgroup $L_{-n+1}$. We shall estimate from below the size of $\varphi(\widetilde{f}_{\underline{i}}\widetilde{a})(J_{-n})$. For this, note that $\widetilde{f}_{\underline{i}}$ may be rewritten as 
 \begin{align*}
 	\widetilde{f}_{\underline{i}}&=\left (\widetilde{g}^{-n+1}\widetilde{h}_0^{i_{n-1}}\widetilde{g}^{n-1}\right )\cdots\left (\widetilde{g}^{-2}\widetilde{h}_0^{i_2}\widetilde{g}^2\right )\left (\widetilde{g}^{-1}\widetilde{h}_0^{i_1}\widetilde{g}\right )\widetilde{h}_0^{i_0}\\
 	&=\widetilde{g}^{-n+1}\widetilde{h}_0^{i_{-n+1}}\widetilde{g}\widetilde{h}_0^{i_{-n+2}}\widetilde{g}\cdots \widetilde{g}\widetilde{h}_0^{i_1}\widetilde{g}\widetilde{h}_0^{i_0}.
 \end{align*}
 Consider the sequence of intervals obtained by applying successively the terms in the latter expression to the interval $\varphi(\widetilde{a})(J_{-n})$. Observe that any such interval stays inside $(\sigma, 1]$: indeed, each application of the generator $\widetilde{g}$ or $\widetilde{h}_0$  moves the interval to the right, and the  final term $\widetilde{g}^{-n+1}$ moves it to the left but is compensated by the previous $n-1$ instances of $\widetilde{g}$ with a positive power. By the mean value theorem and the choice of $\sigma$, at each application of $\widetilde{g}, \widetilde{g}^{-1}$ or $\widetilde{h}_0$, the size of the interval may decrease by at most a factor $1-\varepsilon$. Hence
 \[\left |\varphi(\widetilde{f}_{\underline{i}}\widetilde{a})(J_{-n})\right |\geq (1-\varepsilon)^{nN+2n-2}\left |\varphi(\widetilde{a})(J_{-n})\right |\geq C|J_0|(1-\varepsilon)^{nN+2n-2}\lambda^n,\]
 where we have used \eqref{e-Jn}, and the constant $C$ is $\min_{\xi\in[0,1]} D\varphi(\widetilde{a})(\xi)$. Summing over $\underline{i}$ and using that the intervals are disjoint, we have
 \[1\geq \sum_{\underline{i}\in\{1,\ldots,N\}^n}\left |\varphi(\widetilde{f}_{\underline{i}}\widetilde{a})(J_{-n})\right |\geq C|J_0| (1-\varepsilon)^{-2}\left ((1-\varepsilon)^{2N+2}N\lambda\right )^n,\]
 which is impossible, since by \eqref{e-epsilon-choice} the right-hand side is unbounded as $n$ tends to $+\infty$. This is the desired contradiction. 
 \end{proof}
 
 We will also use the following result, whose proof readily follows from the discussion in  \cite[\S 4.2]{BMNR}.
 \begin{lem} \label{l-upgrade-to-conjugacy}
 Let $\varphi\colon G\to \Diff_0^1([0, 1])$ be an irreducible action. Suppose that $\varphi$ is semi-conjugate to an action $\psi\colon G\to \homeo_0(\R)$, and that there exist elements $a, b\in G$ such that $\psi(a)$ is a homothety $x\mapsto \lambda x$ with $\lambda>1$ and $\psi(b)$ is a translation $x\mapsto x+\alpha$ with $ \alpha \neq 0$. Then $\varphi$ is minimal on $(0, 1)$, and thus conjugate to $\psi$. 
 \end{lem}
 
\begin{proof}[Proof of Theorem \ref{mthm.C1}]
Let $\varphi\colon G\to \Diff^1([0,1])$ be an irreducible action. If by contradiction $\varphi$ to $(0, 1)$ is not semi-conjugate to any affine action, by Theorem \ref{t-main} it is semi-conjugate to a minimal laminar action $\psi\colon G\to \homeo_0(\R)$, whose kernel we denote by $N$. By Proposition \ref{p-plante-ubiquitous} we have that $G/N$ contains a subgroup $ \Z \wr \Z$ that acts via an action which is semi-conjugate to a Plante-like action. If we choose two preimages $\widetilde{h}_0,\widetilde{g}\in G$ of the corresponding generators of $\Z \wr \Z$, the restriction of $\varphi$ to the group $H= \langle \widetilde{h}_0, \widetilde{g} \rangle$ satisfies all assumptions of Proposition \ref{p-plante-C1}, giving a contradiction. The last sentence of the theorem follows from Lemma \ref{l-upgrade-to-conjugacy}.
\end{proof}

\section{The metanilpotent case} \label{s-metanilpotent}

In this section we consider the class of {virtually metanilpotent} groups. Recall that $G$ is virtually metanilpotent if it admits a finite index subgroup $G_0\le G$ which is metanilpotent, meaning that there is a normal subgroup $N\unlhd G_0$ such that $N$ and $G_0/N$ are both nilpotent. 

\begin{rem} \label{r-Malcev}
The class of virtually metanilpotent groups contains in particular metabelian groups and all virtually solvable linear groups, i.e.\ virtually solvable subgroups of $\GL(n, \mathbb{K})$ where $\mathbb{K}$ is a field. Indeed by a classical theorem of Mal'cev \cite{Malcev}, every solvable subgroup of $\GL(n, \mathbb{K})$ has a finite index subgroup which is triangularizable over the algebraic closure of $\mathbb{K}$, and thus is nilpotent-by-abelian. 
\end{rem}

For this class of groups, the results in Section \ref{sec Focal solvable} become much stronger.  Theorem \ref{t-solvable-inductive} specifies as follows. 
\begin{cor}\label{c-metanilpotent-tree}
Let $\Phi \colon G\to \Aut(\Tbb, \treeorder)$ be a d-minimal action of a virtually metanilpotent group on a directed tree. Then $\Phi$ can be horograded by an action by translations $j\colon G\to (\R, +)$. Equivalently, $\Phi $ preserves a compatible $\R$-tree metric. 
\end{cor}
\begin{proof} 
 By Theorem \ref{t-solvable-inductive}, $\Phi $ can be horograded by a minimal action $j\colon G\to \homeo_0(\R)$ which  factors through $G/\Fit(G)$.  The fact that $G$ is virtually metanilpotent implies that $G/\Fit(G)$ is virtually nilpotent (using Lemma \ref{l-Fitting}). 
Recall that every minimal action of a virtually nilpotent group on the line preserves a Radon measure: for countable groups, this is the classical result of Plante \cite{PlanteMeasure} mentioned in the introduction, and for general groups one can see for instance Beklaryan \cite[Theorem B]{Beklaryan}. Since $j$ is minimal, such a measure must be atomless and of full support, hence it can be mapped to the Lebesgue measure by some homeomorphism, and it follows that $j$ is conjugate to an action by translations. The equivalent formulation  in terms of existence of an invariant metric is a consequence of Proposition \ref{prop.Rtreemetric}. \qedhere

\end{proof}
\begin{rem} The reader may notice that the proof above works more generally for groups  with virtual Fitting length $\vf(G)=2$. However one can check that for finitely generated groups, this condition is equivalent to being virtually metanilpotent. 
\end{rem}

As a consequence, Theorem \ref{t-main} translates here as follows. 

\begin{cor}\label{c-metanilpotent-dichothomy}
Any minimal action $\varphi\colon G\to \homeo_0(\R)$ of a finitely generated virtually metanilpotent group is either conjugate to an affine action, or laminar, in which case it can be horograded by an action by translations $j\colon G\to (\R, +)$.
\end{cor}

\subsection{Perturbations of affine actions}
Here we prove Theorem \ref{mthm.affine_rigid}. It will follow from the results above and from some general results on perturbations of affine actions of a finitely generated group.

\begin{prop}\label{p.translations_away}
	Let $G$ be a finitely generated group. Then for any irreducible action $\varphi_0\colon G\to \homeo_0(\R)$ which is semi-conjugate to a non-abelian affine action, there exists an open neighborhood of $\varphi_0$ in $\Homirr(G,\homeo_0(\R))$ which contains no action in the semi-conjugacy class of any action by translations.
\end{prop}

\begin{proof}
	Let $S$ be a finite symmetric generating set of $G$ containing the identity. As $\varphi_0$ is semi-conjugate to a non-abelian affine action $\overline{\varphi}_0\colon G\to \Aff(\R)$, we can choose $s\in S$ such that $\overline{\varphi_0}(s)$ is an expanding homothety. Thus $\varphi_0(s)$ is an expanding pseudo-homothety, and we let $K$ be its compact set of fixed points. Choose $\xi_-< \min K$ and $\xi_+>\max K$. Since $\varphi_0(s^n)(\xi_+)$ tends to $+\infty$, we can choose $n>0$ such that $\varphi_0(s^n)(\xi_+)> \max_{t\in S} \varphi_0(t)(\xi_+)$. Let $\mathcal{U}$ be the set of actions $\psi\in \Homirr(G,\homeo_0(\R))$ satisfying the open conditions
	\[\psi(s^n)(\xi_+)>\max_{t\in S} \psi(t)(\xi_+),\quad \psi(s)(\xi_+)>\xi_+,\quad \psi(s)(\xi_-)<\xi_-.\]
	Then $\mathcal{U}$ is an open neighborhood of $\varphi_0$ and we claim that it satisfies the desired conclusion. Indeed, assume by contradiction that $\psi\in \mathcal{U}$ is positively semi-conjugate to an action  by translations given by a non-trivial homomorphism $\overline{\psi}\colon G\to (\R, +)$. Since $\psi(s)(\xi_-)<\xi_-$ and $\psi(s)(\xi_+)>\xi_+$, we must have $\overline{\psi}(s)=0$. On the other hand, since $S$ is symmetric and $\overline\psi$ non-trivial, there exists $t\in S$ such that $\overline{\psi}(t)>0$. In particular, for such a $t$ we must have $\psi(t)(\xi_+)>\psi(s^n)(\xi_+)$ for every $n\in \Z$, which contradicts the assumption $\psi\in \mathcal{U}$.
\end{proof}

\begin{lem}\label{lem.domination}
	Let $G$ be a finitely generated group, and fix a finite symmetric generating subset $S\subset G$. Let $\varphi\colon G\to\homeo_0(\R)$ be a minimal laminar action, positively horograded by an action by translations $j\colon G\to (\R,+)$. Then, for every element $g\in \ker j$ and $p\in\R$, there exists $h\in S\cup \{sg^\epsilon s^{-1}:s\in S, \epsilon\in \{\pm 1\}\}$ such that \[\varphi(h^{-1})(p)<\varphi(g^n)(p)<\varphi(h)(p)\]
	for every $n\in\Z$. 
\end{lem}
\begin{proof}
	For sake of readability, we resume our notation $\varphi(g)(x)$ for $g.x$, and similar variations, when working with the action $\varphi$.
	Let $(\mathcal{L}, \hor)$ be a positive horograding of $\varphi$ by $j$; without loss of generality, since $G$ is finitely generated, we can assume that the lamination $\mathcal L$ is minimal $\varphi$-invariant (Proposition \ref{prop.minimalamination}). Take an element $g\in \ker j$ and note that by Proposition \ref{p-classification-elements}, $\varphi(g)$ is totally bounded. Fix a point $p\in\R$. If $g.p=p$, it is enough to choose $h\in S$ such that $h.p>p$, which exists since $\varphi$ has no global fixed points. Suppose now that $g.p\neq p$ and let $J$ be the connected component of $\suppphi(g)$ containing $p$. Define $l_0:=\min_{\subseteq}\{l\in\mathcal{L}:J\subseteq l\}$. The fact that $l_0$ is well defined follows from the fact that leaves that contain $J$ are totally ordered and that $\mathcal{L}$ is closed.  Notice that, by $\varphi(g)$-invariance of ${J}$ and the definition of $l_0$, we must have $g.l_0=l_0$. 
	
\begin{claimnum}\label{sublem.fix}
	If $k\in \mathcal L$ is such that $k\subsetneq l_0$ and $g.k=k$, then the intersection $k\cap J$ is empty.
\end{claimnum}
\begin{proof}[Proof of claim]
	Let assume by contradiction that $k\cap J\neq \varnothing$. Note that, after the choice of $l_0$, $k$ cannot contain $J$. Therefore, $k$ must have at least one endpoint in $J$. Since $\varphi(g)$ moves every point in $J$, this implies that $g.k\neq k$.
\end{proof}
	
\begin{claimnum}\label{sublem.related} Suppose that for some $h\in G$ one of the following holds:
\begin{enumerate}
\item\label{iA} either $l_0$ and $h.l_0$ are disjoint, or
\item\label{iB} $h.l_0\subsetneq l_0$ and $gh.l_0=h.l_0$.
\end{enumerate}
Then, up to replacing $h$ by its inverse, it holds that $h^{-1}.p<g^n.p<h.p$ for every $n\in\Z$. 
\end{claimnum}
\begin{proof}[Proof of claim]
	We first observe that both conditions  \eqref{iA} and \eqref{iB} imply that $h.J\cap J=\varnothing$. If \eqref{iA} holds, this is obvious since $J\subseteq l_0$. If \eqref{iB} holds, by Claim \ref{sublem.fix}, we have $h.l_0\cap {J}=\varnothing$ and the same conclusion follows, since $h.J\subseteq h.l_0$. Now, upon replacing $h$ by its inverse, we can assume that $h.J$ lies on the right of $J$  (and thus $h^{-1}.J$ lies on its left). Since $p\in J$ and its $\varphi(g)$-orbit is contained in $J$, the claim follows.
\end{proof}

Let $s\in S$ be a generator such that $j(s)$ is a positive translation. If $s.l_0$ and $l_0$ are disjoint, then we can conclude the proof using Claim \ref{sublem.related}, as condition \eqref{iA} is satisfied. Otherwise, since $j(s)$ is a positive translation, we must have $l_0\subsetneq s.l_0$.
Assume first that $sgs^{-1}.l_0\neq l_0$. In this case, we have \[\hor(sgs^{-1}.l_0)=j(sgs^{-1})(\hor(l_0))=j(g)(\hor(l_0))=\hor(l_0);\]
if $sgs^{-1}.l_0$ and $l_0$ are related by inclusion, Lemma \ref{l-horograding-almost-injective} gives that they are in distinct $\varphi$-orbits, an absurd (note that this is the place where we use that $\mathcal L$ is minimal). Thus, it must be that $sgs^{-1}.l_0$ and $l_0$ are disjoint. Thus, we conclude again using Claim \ref{sublem.related}, as condition \eqref{iA} is satisfied by $h=sgs^{-1}$. It remains to consider the case $sgs^{-1}.l_0=l_0$. For this, rewrite the equality as $gs^{-1}.l_0=s^{-1}.l_0$ and notice that $s^{-1}.l_0\subsetneq l_0$. In this case, we conclude using Claim \ref{sublem.related}, setting $h=s^{-1}$ in condition \eqref{iB}. 
\end{proof}

\begin{prop}\label{prop.focal_away}
Let $G$ be a finitely generated group. Denote by $\mathcal{X}\subset \Homirr(G, \homeo_0(\R))$ the subset of minimal laminar actions of $G$ that can  be horograded by an action by translations, and let $\varphi_0\colon G\to \homeo_0(\R)$ be an action which is semi-conjugate to a non-abelian affine action. Then $\varphi_0$ has a neighborhood $\mathcal{U}$ which contains no action in the semi-conjugacy class of any action in $\mathcal X$.
\end{prop}

\begin{proof}
Let $\varphi_0\colon G\to\homeo_0(\R)$ be an action which is semi-conjugate to a non-abelian affine action $\bar\varphi_0$, fix a point $p\in\R$, and  a finite symmetric generating subset $S\subset G$. Since $G$ is non-abelian, there exists $g\in[G,G]$ such that $\bar\varphi_0(g)$ acts as a translation, and thus $\varphi_0(g)$ acts without fixed points. Up to replace $g$ by its inverse, we can assume $\varphi_0(g)(p)>p$. Thus,  there exists $K\in\N$ such that for any $h$ in the finite subset $F:=S\cup \{sg^\epsilon s^{-1}:s\in S, \epsilon\in \{\pm 1\}\}$ we have \[\varphi_0(g^{-K})(p)<\varphi_0(h)(p)<\varphi_0(g^K)(p).\]
This is an open condition in $\Homirr(G, \homeo_0(\R))$, so we can find a neighborhood $\mathcal U$ of $\varphi_0$ such that the condition above is satisfied by any $\varphi\in \mathcal U$. Assume now that  $\varphi\colon G\to \homeo_0(\R)$ is an action which is positively semi-conjugate to an action   $\bar{\varphi}\in \mathcal{X}$, and let $\tau\colon \R \to \R$ be a non-decreasing  map realizing the semi-conjugacy, in the sense that  $\tau(\varphi(f)(x))=\bar{\varphi}(f)(\tau(x))$ for every $f\in G$ and $x\in \R$. Using Lemma \ref{lem.domination}, we find an element $h\in F$ such that $\bar{\varphi}(h^{-1})(\bar p)<\bar{\varphi}(g^n)(\bar p)<\bar{\varphi}(h)(\bar p)$ for every $n\in \Z$, where $\bar{p}=\tau(p)$. Therefore the analogue condition is satisfied for the action $\varphi$: $\varphi(h^{-1})(p)<\varphi(g^n)(p)<\varphi(h)(p)$ for every $n\in \Z$.
We deduce that $\varphi\notin \mathcal U$.
\end{proof}

\begin{proof}[Proof of Theorem \ref{mthm.affine_rigid}] 
Since $G$ is finitely generated, every $\varphi\in \Homirr(G, \homeo_0(\R))$ is semi-conjugate to an action which is either minimal or cyclic. After Corollary \ref{c-metanilpotent-dichothomy}, and by splitting further the affine actions into the non-abelian ones and actions by translations, we obtain that every $\varphi\in \Homirr(G, \homeo_0(\R))$ is semi-conjugate to either
\begin{enumerate}
\item \label{i-affine} a non-abelian affine action, or
\item \label{i-abelian} an action by translations, or
\item \label{i-focal} a minimal laminar action which can be horograded by an action by translations. 
\end{enumerate} 
By Propositions \ref{p.translations_away} and \ref{prop.focal_away}, every action satisfying \eqref{i-affine} has a neighborhood consisting of actions which satisfy neither \eqref{i-abelian} nor \eqref{i-focal}.  Thus the set of such actions is open.
\end{proof}

\subsection{A non-metanilpotent example}\label{ssc.ZwrZwrZ}
The goal of this subsection is to show that the conclusion of Theorem \ref{mthm.affine_rigid} fails for solvable groups which are not virtually metanilpotent. More precisely, we will construct a finitely generated 3-step solvable group $G$ and a sequence $(\varphi_n)\subset \Homirr(G, \homeo_0(\R))$ of minimal actions which are not semi-conjugate to any affine action, but converge to a limit $\varphi\in \Homirr(G, \homeo_0(\R))$ which is conjugate to a non-abelian affine action.

As a basis for our construction,  let $B\leq \Aff(\R)$, be a finitely generated non-abelian group of affine transformations. A precise choice of such an $B$ will be irrelevant; for definiteness the reader may take $B$ to be the group generated by the two transformations 
\[ x\mapsto a x,   \quad  x\mapsto x+1  \quad (a>0,a\neq 1).\]
We denote by $j \colon B\to \homeo_0(\R)$ the action given by the standard inclusion.

Let $G$ be the wreath product $G=\Z \wr B= \bigoplus_B\Z\rtimes B$, as defined in \S\ref{s-wreath}. Extending the notation from \S \ref{ssc.plante}, we will simply write $L$ for the direct sum $\bigoplus_B\Z$, and identify it with the group of finitely supported configurations $f\colon B\to \Z$. The semi-direct product is taken with respect to the left-regular action $\lambda \colon B\to \Aut(L)$, given by $\lambda_b(f)(k)=f(b^{-1}k)$. For later use, we also denote by $\rho\colon B\to \Aut( L)$ the right-regular action, given by $\rho_b(f)(k)=f(kb)$. Note that these two actions commute.

We denote by $\pi_B\colon G\to B$ the projection to the quotient and by $(\delta_b)_{b\in B}$ the standard basis of $L$, where $\delta_b\in L $ is the configuration taking the value 1 on $b$ and $0$ elsewhere. The identically zero configuration of $L$ is denoted as $\bar{0}$. Finally for $f\in L$ we write $\supp (f)=\{b\in B : f(b)\neq 0\}$.

Chose a point $\xi_0\in \R$ which has a free orbit for the natural affine action $j$ of $B$, and let $<_B$ be the associated left-order on $B$, given by $b_1<_Bb_2$ if $j(b_1)(\xi_0)<j(b_2)(\xi_0)$. We can use this order  to construct an action of $G$ on $\R$ in a way which is analogous to the Plante-like action from \S \ref{ssc.plante} (see the discussion in \S \ref{s-wreath}, and \cite[Example 8.1.8]{BMRT} for more details). Namely, consider the order $\prec$ on $L$ given by $f_1\prec f_2$ if $f_1(b_*)<f_2(b_*)$, where $b_*=\max_{<_B} \{b : f_1(h)\neq f_2(b)\}$. Let $\alpha \colon G\to \Aut( L, \prec)$ be the ``affine'' action obtained by letting $B$ act via the left-regular action and $L$ act on itself by translations. Explicitly, for $g=(r, b)\in L\rtimes B=G$ and  $ f\in L$, we have
\[\alpha(g)(f)=r+\lambda_b(f).\]
It is straightforward to check that this action preserves the order $\prec$ constructed above. By taking the dynamical realization of the action $\alpha \colon G\to \Aut (L, \prec)$, we obtain an action 
$\psi\colon G\to \homeo_0(\R)$. This action is minimal and laminar (and is horograded by the action $j\circ \pi_B\colon G\to \homeo_0(\R)$), we refer to   \cite[Example 8.1.8]{BMRT} for a proof.

We now perturb the $\psi$-action of $G=\Z\wr B$, to build our desired counterexample. Let $\iota \colon L\to \R$ be the associated equivariant order-preserving map from the dynamical realization $\psi$. We suppose that $\iota(\overline{0})=0$. Then $0$ is a global fixed point for $\psi(B)$, and it is moved by all elements of $\psi(L)$. Denote by $I_b$ the connected component of the support of $\psi(\delta_b)$ containing 0. Then each $I_{b}$ is a bounded open interval and we have $I_{b_1}\subset I_{b_2}$ if $b_1<_B b_2$.  Moreover since $0$ is fixed by $\psi(B)$  and $b\delta_cb^{-1}=\delta_{bc}$, we have $\psi(b)(I_c)=I_{bc}$. In particular if we denote by $p_b$ the rightmost point of $I_b$, the map $b\mapsto p_b$ is an order-preserving, $B$-equivariant embedding of $(B,<_B)$ into $\R$.   Finally we observe that the point $p_b$ is fixed by $\psi(\delta_c)$ for every $c<_B b$. Thus if $f\in L$ is such that $\max_{<_B} \supp (f)<_B b$, then $\psi(f)(p_b)=(p_b)$. 

Let now $t\in B$ be any element such that $j(t)$ is a positive translation. Let $\tau\colon G\to G$ be the automorphism
\[\tau(f, b)=(\rho_t(f), b), \quad (f, b)\in L \rtimes B.\]
The fact that $\tau$ is an automorphism follows from the fact that $\rho_t$ commutes with the left-regular action of $B$. For every $n\in \Z$ set $\psi_n=\psi \circ \tau^n$. Note that, in the same way as $\psi$, every $\psi_n$ is minimal and laminar. We also point out to the reader that $\tau$ is an exterior automorphism of $G$, and the actions $\psi_n$ are all pairwise non-conjugate. We will prove the following.

\begin{prop}\label{prop ejemplo no acumulado}
With notation as above, there exists a sequence of homeomorphisms $(s_n)\subset \homeo_0(\R)$ such that the sequence of conjugate actions $ \varphi_n(g)=s_n\psi_n(g)s_n^{-1}$ has a subsequence which converges to a limit $\varphi\in \Homirr(G, \homeo_0(\R))$, which is positively conjugate to the non-abelian affine action $j\circ \pi_B$. 
\end{prop}
\begin{proof}
Before discussing the  details of the proof we recall here a fundamental consequence of the work of Deroin, Kleptsyn, Navas, and Parwani \cite{DKNP} on random walk on subgroups of $\homeo_0(\R)$. As explained in \cite[\S 14]{BMRT}, it follows from their work that for every finitely generated group $G$, one can find a subset $\mathscr{H}\subset \Homirr(G,\homeo_0(\R))$ having the following properties:
\begin{enumerate}
	\item $\mathscr{H}$ is compact,
	\item $\mathscr{H}$ contains a representative of any positive semi-conjugacy class of irreducible action of $G$, and moreover this representative is either minimal or cyclic,
	\item $\mathscr{H}$ is closed under conjugation by translations, and conversely any two representatives of the same positive semi-conjugacy class are conjugate by a translation.\footnote{More precisely, one can choose $\mathscr{H}$ to be the space of so-called \emph{normalized $\mu$-harmonic actions}, associated with the choice of a symmetric finitely supported probability measure $\mu$ on $G$. However, the precise nature of $\mathscr{H}$ will be irrelevant for our purposes, and we will use only the three properties listed above. Note that these properties characterize $\mathscr{H}$ up to homeomorphism, see \cite[\S 3]{BMRT-Deroin}.}
\end{enumerate}

Write $p_0:=p_{\id_B}$. As $\psi_n$ is minimal for any $n\in \N$, by the last two properties of the  space $\mathscr{H}$ we can find a sequence $(s_n)$ of homeomorphisms such that $s_n(p_{0})=p_{0}$, and such that $\varphi_n:={}^{s_n}\psi_n\in \mathscr{H}$. As $\mathscr{H}$ is compact, upon extracting a subsequence we can suppose that $\varphi_n$ converges to a limit $\varphi\in \mathscr{H}$. In particular, $\varphi$ is either minimal or cyclic.  We claim that $\varphi$ is positively conjugate to $j\circ \pi_B$. To see this,  fix $g=(f, b) \in G$, with $b=\pi_B(g)$ and compute
\begin{align*}
	\varphi_n(g)(p_0)&=s_n  \psi_n(g)  s_n^{-1} (p_0)=s_n\psi_n(g)(p_0)\\&=s_n\psi(\rho_{t^n}(f), b)(p_0)=s_n\psi(\rho_{t^n}(f)) (p_b).\end{align*}
Here we have used that the map $b\mapsto p_b$ is $B$-equivariant, as observed above. Now note that $\supp(\rho_{t^n}(f))=\supp(f)t^{-n}$. By the choice of $t$, the element $b_n:=\max_{<_B}\supp(f)t^{-n}$ tends to $-\infty$ in $(B, <_B)$. In particular if $n$ is large enough we have $b_n<_Bb$. As observed earlier, this implies that $\psi(\rho_{t^n} (f))(p_b)=p_b$. Thus, for $n$ large enough we have
\[\varphi_n(g)(p_0)=s_n(p_b)=s_n(p_{\pi_B(g)}).\]
Now, fix $g_1, g_2\in G$. If $n$ is large enough, the previous computation holds for both $g_1$ and $g_2$. Therefore the following conditions are equivalent for sufficiently large $n$:
\begin{enumerate}
	\item $\varphi_n(g_1)(p_0)\le \varphi_n(g_2)(p_0)$,
	\item $s_n(p_{\pi_B(g_1)})\le s_n(p_{\pi_B(g_2)})$,
	\item $\pi_B(g_1) \le _B \pi_B(g_2).$
\end{enumerate}
For the equivalence between the last two conditions, we use that $s_n$ is a homeomorphism and that the map $b\mapsto p_b$ is order preserving.
Taking the limit, we get that if $\pi_B(g_1) \le_B \pi_B(g_2)$, then $\varphi(g_1)(p_0)\le \varphi(g_2)(p_0)$.
We deduce that the map $\sigma\colon (B, <_B)\to (\R, <)$ given by {$\sigma(b)=\varphi((\overline{0}, b))(p_0)$} is non-decreasing and  $G$-equivariant, where $G$ acts on   $(B, <_B)$ by translations via  $\pi_B$, and on $\R$ via $\varphi$. Note that the dynamical realization of the action on $(B, <_B)$ is precisely $j\circ \pi_B$, and thus it is minimal. By Lemma \ref{l-dynamical-realisation}, we deduce that $\sigma$ is injective. In particular $\varphi$ cannot be cyclic, since the ordered space $(B, <_B)$ is not isomorphic to  $(\Z, <)$. Thus $\varphi$ is minimal, and  Lemma \ref{l-dynamical-realisation}  again implies that it is conjugate to $j\circ \pi_B$. 
\end{proof}

\bibliography{biblio.bib}

\noindent\textit{Joaqu\'in Brum\\
	IMERL, Facultad de Ingenier\'ia, Universidad de la República, Uruguay\\
	Julio Herrera y Reissig 565, Montevideo, Uruguay\\}
\href{mailto:joaquinbrum@fing.edu.uy}{joaquinbrum@fing.edu.uy}

\smallskip

\noindent\textit{Nicol\'as Matte Bon\\
	CNRS \&
	Institut Camille Jordan (ICJ, UMR CNRS 5208)\\
	Universit\'e de Lyon\\
	43 blvd.\ du 11 novembre 1918,	69622 Villeurbanne,	France\\}
\href{mailto:mattebon@math.univ-lyon1.fr}{mattebon@math.univ-lyon1.fr}

\smallskip

\noindent\textit{Crist\'obal Rivas\\
Departamento de Matem\'aticas\\
Universidad de Chile\\
Las Palmeras 3425, Ñuñoa, Santiago, Chile\\}
\href{mailto:cristobalrivas@u.uchile.cl}{cristobalrivas@u.uchile.cl}

\smallskip

\noindent\textit{Michele Triestino\\
	Institut de Math\'ematiques de Bourgogne (IMB, UMR CNRS 5584)  \& Institut Universitaire de France\\
	Universit\'e de Bourgogne\\
	9 av.~Alain Savary, 21000 Dijon, France\\}
\href{mailto:michele.triestino@u-bourgogne.fr}{michele.triestino@u-bourgogne.fr}

\end{document}